\documentclass[11pt,hyp,]{nyjm}
\usepackage{hyperref}
\hypersetup{nesting=true,debug=true,naturalnames=true}
\usepackage{graphicx,amssymb,upref}

\let\<\langle
\let\>\rangle
\newcommand{\doi}[1]{doi:\,\href{http://dx.doi.org/#1}{#1}}

\let\uml\"

\usepackage{mathrsfs}
\usepackage{amssymb}
\usepackage{amsmath}
\usepackage{dsfont}
\usepackage{enumerate}

\newtheorem{theorem}{Theorem}[section]
\newtheorem{lemma}[theorem]{Lemma}

\newtheorem{definition}[theorem]{Definition}
\newtheorem{remark}[theorem]{Remark}

\newtheorem{proposition}[theorem]{Proposition}

\newtheorem{problem}[theorem]{Problem}

\numberwithin{equation}{section}

\allowdisplaybreaks

\title[The random Kakutani fixed point theorem]{The random Kakutani fixed point theorem in random normed modules}

\author{Qiang Tu}
\address{School of Mathematics and Statistics, Central South University,
	Changsha {\rm 410083}, China}
\email{qiangtu126@126.com}
\thanks{This work was supported by the National Natural Science Foundation of China (Grant Nos.12371141, 12426645, 12426654) and the Natural Science Foundation of Hunan Province of China (Grant No.2023JJ30642).}

\author{Xiaohuan Mu}
\address{School of Mathematics and Statistics, Central South University,
	Changsha {\rm 410083}, China}
\email{xiaohuanmu@163.com}

\author{Tiexin Guo*}
\address{School of Mathematics and Statistics, Central South University,
	Changsha {\rm 410083}, China}
\email{tiexinguo@csu.edu.cn}
\thanks{*Corresponding author}
\author{Guang Yang}
\address{School of Mathematics and Statistics, Central South University,
	Changsha {\rm 410083}, China}
\email{guuangyang@163.com}

\author{Yuanyuan Sun}
\address{School of Mathematics and Statistics, Central South University,
	Changsha {\rm 410083}, China}
\email{yuanyuansun1205@163.com}



\begin{document}
	
	\begin{abstract}
	Based on the recently developed theory of random sequential compactness, we prove the random Kakutani fixed point theorem in random normed modules: if $G$ is a random sequentially compact $L^{0}$-convex subset of a random normed module, then every $\sigma$-stable $\mathcal{T}_{c}$-upper semicontinuous mapping $F:G\rightarrow 2^{G}\setminus \{\emptyset\}$ such that $F(x)$ is closed and $L^{0}$-convex for each $x\in G$, has a fixed point. This is the first fixed point theorem for set-valued mappings in random normed modules, providing a random generalization of the classical Kakutani fixed point theorem as well as a set-valued extension of the noncompact Schauder fixed point theorem established in [Guo et al., Math. Ann. 391(3), 3863--3911 (2025)].
	\end{abstract}
	\maketitle
	\tableofcontents

	%###############################################################################
	%###############################################################################
	\section{Introduction}\label{sec.1}
	%###############################################################################
	%###############################################################################

    \par 
    The celebrated Kakutani fixed point theorem \cite{Kakutani1941} states that any upper semicontinuous set-valued mapping from a compact convex subset of the $n$-dimensional Euclidean space $\mathbb{R}^{n}$ into itself, with closed convex values, has a fixed point. This theorem extends the Brouwer fixed point theorem \cite{Brouwer1912} from continuous single-valued mappings to upper semicontinuous set-valued mappings, and has subsequently been generalized to Banach spaces by Bohnenblust and Karlin \cite{BK1950}, and to locally convex spaces by Glicksberg \cite{Gli1952} and Fan \cite{Fan1952}. The Kakutani fixed point theorem and its generalizations have become powerful tools in game theory \cite{Aub1979,Bor1985,Deb1959,Nash1951',SV1991}, optimal control theory \cite{LO1955}, differential equations \cite{Cel1970} and various other areas of mathematics \cite{GD2003}, for example, Nash \cite{Nash1951'} provided a concise one-page proof of his equilibrium theorem by reformulating the problem of the existence of equilibrium points as the problem of the existence of Kakutani fixed points.

    \par 
    Recently, Guo et al. \cite{GWXYC2025} established a noncompact Schauder fixed point theorem and applied it to prove a dynamic Nash equilibrium theorem, while Ponosov \cite{Ponosov1988,Ponosov2021,Ponosov2022} proved a stochastic version of the Brouwer fixed point theorem and employed it to the study of stochastic differential equations. Notably, \cite[Theorem 5.4]{TMGC2025} shows that this stochastic version of the Brouwer fixed point theorem is equivalent to a special case of the random Brouwer fixed point theorem established in \cite{GWXYC2025}, where the latter is naturally a special case of the noncompact Schauder fixed point theorem. With the aim of providing a new tool for the development of random functional analysis \cite{Guo2010}, dynamic Nash equilibrium theory \cite{GWXYC2025} and stochastic differential equations \cite{Ponosov2021,Ponosov2022}, this paper proves the random Kakutani fixed point theorem in random normed modules, a result that not only extends the noncompact Schauder fixed point theorem \cite{GWXYC2025} from continuous single-valued mappings to upper semicontinuous set-valued mappings but also generalizes the classical result of Bohnenblust and Karlin \cite{BK1950} from Banach spaces to random normed modules. Moreover, it is the first fixed point theorem for set-valued mappings in random normed modules and provides a basic result for future developments in fixed point theory in random functional analysis.

    \par 
    Random normed modules, a central framework of random functional analysis, were independently introduced by Guo in connection with the idea of randomizing classical space theory \cite{Guo1992,Guo1993} and by Gigli in connection with nonsmooth differential geometry on metric measure spaces \cite{Gigli2018} (see also \cite{BPS2023,CLPV2025,CGP2025,GLP2025,LP2019,LPV2024}  for related advances). Different from the situation in classical normed spaces, the $L^{0}$-norm of a random normed module induces two different topologies. One is the $(\varepsilon,\lambda)$-topology, which is a typical metrizable locally nonconvex linear topology. The other is the locally $L^{0}$-convex topology \cite{FKV2009}, which is stronger than the $(\varepsilon,\lambda)$-topology but generally not linear. The two topologies both have their respective advantages and disadvantages in theoretical investigations and financial applications.  To combine their strengths,  Guo \cite{Guo2010} introduced the notion of a $\sigma$-stable set and established the connection between some basic results derived from two topologies. This work considerably advanced the development of random functional analysis and its applications \cite{Guo2024,GWXYC2025,GWYZ2020,GZZ2014,GZWG2020,GZWYZ2017}. Over the past decade, one of the central topics in random functional analysis has been to overcome the challenge due to noncompactness: closed $L^{0}$-convex subsets of a random normed module --- which frequently arise in both theory and financial applications --- are generally not compact under the $(\varepsilon,\lambda)$-topology \cite{Guo2008}, and hence also not compact under the locally $L^{0}$-convex topology. Consequently, classical compactness arguments are no longer applicable. Motivated by the studies on $\sigma$-stability \cite{GMT2024,GWYZ2020} and the randomized version of the Bolzano-Weierstrass theorem \cite{KS2001}, Guo et al. \cite{GWXYC2025} introduced and systematically studied the notion of random sequential compactness, which is a genuine generalization of classical sequential compactness. Specifically, a sequence in a random sequentially compact set $G$ may not admit any subsequence that converges in the $(\varepsilon,\lambda)$-topology to some point of $G$ but always admits a random subsequence that does. This property has inspired a series of subsequent works on fixed point theorems in random functional analysis \cite{MTG2025,TMG2024,TMGC2025,WG2024}, and, in particular, it allows us to overcome the main challenge in establishing the random Kakutani fixed point theorem in random normed modules.

    \par 
    The success of the proof of the noncompact Schauder fixed point theorem in \cite{GWXYC2025} lies mainly in presenting a proper randomization of the classical Schauder projection method, for which two essential tools were developed. The first is the random Hausdorff theorem \cite[Theorem 3.3]{GWXYC2025}, which states that a $\sigma$-stable set is random sequentially compact if and only if it is complete with respect to the $(\varepsilon,\lambda)$-topology and random totally bounded. The second is the method of introducing random Schauder projections \cite[Lemma 4.8]{GWXYC2025}, allowing the construction of approximating mappings (see \cite[Lemma 4.9]{GWXYC2025} and the proof of \cite[Theorem 2.12]{GWXYC2025} for details). Motivated by these developments, as well as by the classical work of Nikaido \cite{Nikaid1968}, who provided a new proof of the Kakutani fixed point theorem in $\mathbb{R}^{n}$ based on the Schauder projection method, this paper develops an approach to establishing the random Kakutani fixed point theorem in random normed modules. Although our proof is motivated from the ideas of \cite{Nikaid1968} and \cite{GWXYC2025}, we remain to overcome the following three challenges:
    \begin{enumerate}[(1)]
    	\item Random total boundedness is considerably more involved than classical total boundedness. The notion of random total boundedness was introduced by means of the notion of a stably finite set, see \cite{GMT2024} or part (2) of Definition \ref{defn.stably sequentically compactness} of this paper for details. However, a stably finite set is generally neither finite nor even countable. Consequently, the method used in \cite{Nikaid1968} cannot be directly applied. Besides, the complicated stratification structures peculiar to a random normed module must be considered, as shown in the key equation (\ref{eq.Kakutani1}). 
    	\item An upper semicontinuous set-valued mapping presents greater challenges than a continuous single-valued mapping. In the random Kakutani fixed point theorem, directly constructing random Schauder projections based on a stably finite random $\varepsilon$-net, as done in \cite{GWXYC2025}, is not enough to complete the proof. To overcome this difficulty, we first establish Lemma \ref{lemm.projection}, which allows us to carefully select elements from the images of the set-valued mapping on the stably finite random $\varepsilon$-net to construct continuous single-valued mappings. By applying the noncompact Schauder fixed point theorem to these mappings, we obtain fixed points that in turn generate a stable sequence compatible with the upper semicontinuous set-valued mapping.
    	\item Topological arguments in our proof are different from those in \cite[Theorem 2.12]{GWXYC2025}. Random sequential compactness is defined with respect to the $(\varepsilon,\lambda)$-topology, and the proof of the noncompact Schauder fixed point theorem \cite[Theorem 2.12]{GWXYC2025} relies primarily on arguments under this topology. In contrast, our proof of the random Kakutani fixed point theorem is carried out mainly under the locally $L^{0}$-convex topology, which requires the notion of stably sequential compactness \cite{GMT2024} formulated in this topology. Consequently, our proof depends on \cite[Theorem 2.21]{GMT2024}, which establishes the equivalence between a $\sigma$-stable random sequentially compact set and a stably sequentially compact set. 
    \end{enumerate}

    \par The remainder of this paper is organized as follows. Section \ref{sec.2} introduces some preliminaries and further presents the main result --- the random Kakutani fixed point theorem in random normed modules (namely, Theorem \ref{thm.Kakutani}). Section \ref{sec.3} is devoted to the proof of Theorem \ref{thm.Kakutani}. Finally, Section \ref{sec.4} concludes this paper with some remarks and two open problems related to the present study.

    %###############################################################################
    %###############################################################################
    \section{Preliminaries and main result}\label{sec.2}
    %###############################################################################
    %###############################################################################

    \par 
    Throughout this paper, $\mathbb{K}$ denotes the scalar field $\mathbb{R}$ of real numbers or $\mathbb{C}$ of complex numbers, $\mathbb{N}$ the set of positive integers, $(\Omega,\mathcal{F},P)$ a given probability space, $L^0(\mathcal{F},\mathbb{K})$ the algebra of equivalence classes of $\mathbb{K}$-valued random variables on $(\Omega,\mathcal{F},P)$, $L^0(\mathcal{F},\mathbb{N})$ the set of equivalence classes of $\mathbb{N}$-valued random variables on $(\Omega,\mathcal{F},P)$, $L^0(\mathcal{F}):=L^0(\mathcal{F},\mathbb{R})$ and $\bar{L}^0(\mathcal{F})$ the set of equivalence classes of extended real-valued random variables on $(\Omega,\mathcal{F},P)$.

    \par  
    For any $A, B\in \mathcal{F}$,  we will always use the corresponding lowercase letters $a$ and $b$  for the equivalence classes $[A]$ and $[B]$ (two elements $C$ and $D$ in $\mathcal{F}$ are said to be equivalent if $P[(C\setminus D)\cup(D\setminus C)]=0$), respectively. Let $B_{\mathcal{F}}=\{a=[A]:A\in \mathcal{F}\}$, $1=[\Omega]$, $0=[\emptyset]$, $a\wedge b=[A\cap B]$, $a\vee b=[A\cup B] $ and $a^{c}=[A^{c}]$, where $A^c$ denotes the complement of $A$, then $(B_{\mathcal{F}},\wedge,\vee,^{c},0,1)$ is a complete Boolean algebra, namely, a complete complemented distributive lattice (see \cite{Kop1989} for details). Specifically, $B_{\mathcal{F}}$ is called the measure algebra associated with $(\Omega,\mathcal{F},P)$.

    \par 
    It is well known from \cite{DS1958} that $\bar{L}^0(\mathcal{F})$ is a complete lattice under the partial order $\xi \leq \eta$ iff $\xi^0(\omega) \leq \eta^0(\omega)$ for almost all $\omega \in \Omega$ (briefly, $\xi^0(\omega) \leq \eta^0(\omega)$ a.s.), where $\xi^0$ and $\eta^0$ are arbitrarily chosen representatives of $\xi$ and $\eta$ in $\bar{L}^0(\mathcal{F})$, respectively. In particular, the sublattice $(L^0(\mathcal{F}),\leq)$ is a Dedekind complete lattice.

    \par 
    As usual, for $\xi,\eta\in\bar{L}^{0}(\mathcal{F})$, $\xi< \eta$  means $\xi\leq \eta$ and $\xi\neq \eta$, whereas, for any $a\in B_{\mathcal{F}}$, $\xi< \eta$ on $a$ means $\xi^{0}(\omega)< \eta^{0}(\omega)$  for almost all $\omega\in A$, where $A$, $\xi^0$ and $\eta^{0}$ are arbitrarily chosen representatives of $a$, $\xi$ and $\eta$, respectively. Moreover, we denote $L_{+}^{0}(\mathcal{F}):=\{\xi\in L^{0}(\mathcal{F}):\xi\geq 0\}$ and $L_{++}^{0}(\mathcal{F}):=\{\xi\in L^{0}(\mathcal{F}):\xi>0~\text{on}~1\}$.

    \par For any $\xi,\eta\in \bar{L}^{0}(\mathcal{F})$, we use $(\xi=\eta)$ for the equivalence class of $\{\omega\in \Omega:\xi^{0}(\omega)=\eta^{0}(\omega)\}$,
    where $\xi^{0}$ and $\eta^{0}$ are arbitrarily chosen representatives of $\xi$ and $\eta$, respectively. Similarly, one can understand $(\xi<\eta)$ and $(\xi\leq\eta)$.

    \par 
    Definition \ref{defn. RN module} below is adopted from \cite{Guo1992,Guo1993} by following the traditional nomenclature of random metric spaces and random normed spaces (see \cite[Chapters 9 and 15]{SS2005}).

    \begin{definition}[\cite{Guo1992,Guo1993}]\label{defn. RN module}
    An ordered pair $(E, \|\cdot\|)$ is called a random normed module (briefly, an $RN$ module) over $\mathbb{K}$ with base $(\Omega, \mathcal{F}, P)$ if $E$ is a left  module over the algebra $L^{0}(\mathcal{F},\mathbb{K})$ (briefly, an $L^{0}(\mathcal{F},\mathbb{K})$-module) and $\|\cdot\|$ is a mapping from $E$ to  $L^0_+(\mathcal{F})$ such that the following conditions are satisfied:
    	\begin{enumerate} [(1)]
    		\item $\|\xi x\|= |\xi| \|x\|$ for any $\xi \in L^0(\mathcal{F}, \mathbb{K})$ and any $x\in E$;
    		\item $\|x+y\|\leq \|x\|+ \|y\|$ for any $x$ and $y$ in $E$;
    		\item $\|x\|=0$ implies $x=\theta$ (the null element of $E$). 
    	\end{enumerate}
    As usual, $\|\cdot\|$ is called the $L^0$-norm on $E$.
    \end{definition}

    \par  It should be mentioned that the notion of an $L^0$-normed $L^0$-module, which is equivalent to that of an $RN$ module, was independently introduced by Gigli  in \cite{Gigli2018} for the study of nonsmooth differential geometry on metric measure spaces, where the $L^0$-norm was called the pointwise norm.

    \par 
    When $(\Omega,\mathcal{F},P)$ is trivial, namely, $\mathcal{F}=\{\emptyset,\Omega\}$, an $RN$ module $(E,\|\cdot\|)$ over $\mathbb{K}$ with base $(\Omega,\mathcal{F},P)$ reduces to an ordinary normed space over $\mathbb{K}$. The simplest nontrivial $RN$ module is $(L^{0}(\mathcal{F},\mathbb{K}),|\cdot|)$,  where $|\cdot|$ is the absolute value mapping.

    \par 
    For an $RN$ module $(E,\|\cdot\|)$, the $L^{0}$-norm $\|\cdot\|$ can induce two topologies. The first topology is the $(\varepsilon,\lambda)$-topology, whose definition originates from Schweizer and Sklar’s work on probabilistic metric spaces \cite{SS2005}.

    \begin{definition}[\cite{Guo2010}]\label{defn.topology1}
    Let $(E,\|\cdot\|)$ be an $RN$ module over $\mathbb{K}$ with base $(\Omega,\mathcal{F},P)$. For any real numbers $\varepsilon$ and $\lambda$ with $\varepsilon>0$ and $0< \lambda <1$, let $N_{\theta}(\varepsilon, \lambda)=\{x\in E: P\{\omega\in \Omega:\|x\|(\omega)<\varepsilon\}>1-\lambda\}$, then  $\mathcal{U}_{\theta}:=\{N_{\theta}(\varepsilon, \lambda): \varepsilon>0,0<\lambda <1\}$ forms a local base of some metrizable linear topology for $E$,  called the $(\varepsilon, \lambda)$-topology, denoted by $\mathcal{T}_{\varepsilon, \lambda}$.	
    \end{definition}

    \par 
    The $(\varepsilon,\lambda)$-topology is an abstract generalization of the topology of convergence in probability. More precisely, a sequence $\{x_{n}, n \in \mathbb{N}\}$ in an $RN$ module converges to $x$ in this topology if and only if $\{\|x_{n} -x\|, n \in \mathbb{N}\}$ converges to $0$ in probability. This topology is natural and convenient, for example, the development of nonsmooth differential geometry on metric measure spaces is often carried out under it \cite{BPS2023,CLPV2025,CGP2025,GLP2025,LP2019,LPV2024}. Moreover, $(E,\mathcal{T}_{\varepsilon,\lambda})$ is a metrizable topological module over the topological algebra $(L^{0}(\mathcal{F},\mathbb{K}),\mathcal{T}_{\varepsilon,\lambda})$ \cite{Guo2010}.

    \par 
    The second topology is the locally $L^{0}$-convex topology, which can ensure most of the $L^{0}$-convex sets in question to have nonempty interiors and make it possible to establish the continuity and subdifferentiability theorems for $L^{0}$-convex functions, see \cite{FKV2009,GZWYZ2017} for details.

    \begin{definition}[\cite{FKV2009}]\label{defn.topology2}
    Let $(E,\|\cdot\|)$ be an $RN$ module over $\mathbb{K}$ with base $(\Omega,\mathcal{F},P)$. For any $x\in E$ and any $r\in L^{0}_{++}(\mathcal{F})$, let $B(x,r)=\{y\in E:\|y-x\|<r~\text{on}~1\}$, then $\{B(x,r):x\in E, r\in L_{++}^{0}(\mathcal{F})\}$ forms a base for some Hausdorff topology on $E$, called the locally $L^{0}$-convex topology induced by $\|\cdot\|$, denoted by $\mathcal{T}_{c}$.
    \end{definition}

    \par 
    Recall that a nonempty subset $G$ of an $RN$ module $(E,\|\cdot\|)$ is said to be $L^{0}$-convex if $\xi x+(1-\xi)y\in G$ for any $x,y\in G$ and any $\xi\in L_{+}^{0}(\mathcal{F})$ with $0\leq\xi\leq 1$; $L^{0}$-absorbent if for any $x\in E$ there exists $\delta\in L_{++}^{0}(\mathcal{F})$ such that $\lambda x\in G$ for any $\lambda\in L^{0}(\mathcal{F},\mathbb{K})$ with $|\lambda|\leq \delta$; $L^{0}$-balanced if $\lambda x\in G$ for any $x\in G$ and any $\lambda\in L^{0}(\mathcal{F},\mathbb{K})$ with $|\lambda|\leq 1$. It is clear that $B(\theta,r):=\{y\in E:\|y\|<r~\text{on}~1\}$ is $L^{0}$-convex, $L^{0}$-absorbent and $L^{0}$-balanced for any $r\in L_{++}^{0}(\mathcal{F})$. It is well known that $\{B(\theta,r): r\in L_{++}^{0}(\mathcal{F})\}$ forms a local base for the locally $L^{0}$-convex topology $\mathcal{T}_{c}$ on $E$.

    \par 
    For an $RN$ module $(E,\|\cdot\|)$,  $\mathcal{T}_{c}$ is stronger than $\mathcal{T}_{\varepsilon,\lambda}$, but $\mathcal{T}_{c}$ is not a linear topology in general since scalar multiplication is not necessarily continuous under $\mathcal{T}_{c}$. Consequently, $(E,\mathcal{T}_{c})$ is only a topological module over the topological ring $(L^{0}(\mathcal{F},\mathbb{K}),\mathcal{T}_{c})$, see \cite{FKV2009,Guo2010} for details.

    \par  
    For clarity and convenience in presenting the key notion of $\sigma$-stable sets, we first recall some basic notions concerning measure algebras and regular $L^{0}(\mathcal{F},\mathbb{K})$-modules.

    \par 
    From now on, for any $a\in B_{\mathcal{F}}$, we always use $I_{a}$ to denote the equivalence class of $I_{A}$, where $A$ is an arbitrarily chosen representative of $a$ and  $I_{A}$ denotes the characteristic function of $A$ (namely, $I_{A}(\omega)=1$ if $\omega\in A$ and $I_{A}(\omega)=0$ otherwise). For a subset $H$ of a complete lattice (e.g., $B_{\mathcal{F}}$), we use $\bigvee H$ and $\bigwedge H$ to denote the supremum and infimum of $H$, respectively.

    \par 
    As usual, for any $a,b\in B_{\mathcal{F}}$, $a>b$ means $a\geq b$ and $a\neq b$. A subset $\{a_{i}:i\in I\}$ of $B_{\mathcal{F}}$ is called a partition of unity if $\bigvee_{i\in I} a_{i} =1$ and $a_{i} \wedge a_{j} = 0$ for any $i,j\in I$ with $i\neq j$. The collection of all such partitions is denoted by $p(1)$. For any $\{a_{j}:j\in J\}\in p(1)$, it is clear that the set $\{j\in J : a_{j}>0\}$ is at most countable.

    \par 
    An  $L^{0}(\mathcal{F},\mathbb{K})$-module $E$ is said to be regular if $E$ has the following property: for any given two elements $x$ and $y$ in $E$, if there exists $\{a_{n},n\in \mathbb{N}\}\in p(1)$ such that $I_{a_{n}}x=I_{a_{n}}y$ for each $n\in \mathbb{N}$, then $x=y$.

    \par 
    In the remainder of this paper, all the $L^{0}(\mathcal{F},\mathbb{K})$-modules under consideration are assumed to be regular. This restriction is not excessive, since all random normed modules are regular.

    \begin{definition}[\cite{Guo2010}]\label{defn.stable set}
    Let $E$ be an $L^0(\mathcal{F},\mathbb{K})$-module and $G$ be a nonempty subset of $E$. $G$ is said to be finitely stable if $I_{a}x+I_{a^{c}}y\in G$ for any  $x,y\in G$ and any $a\in B_{\mathcal{F}}$. $G$ is said to be $\sigma$-stable (or to have the countable concatenation property in the original terminology of  \cite{Guo2010}) if for each sequence $\{x_{n}, n\in \mathbb{N} \}$ in $G$ and each $\{a_{n},n\in\mathbb{N}\} \in p(1)$, there exists some $x\in G$ such that  $I_{a_{n}}x=I_{a_{n}}x_{n}$ for each $n\in \mathbb{N}$ ($x$ is unique since $E$ is assumed to be regular, usually denoted by $\sum_{n=1}^{\infty}I_{a_{n}}x_{n}$,  called the countable concatenation of $\{x_n,n\in\mathbb{N}\}$ along $\{a_n,n\in \mathbb{N}\}$). By the way, if $G$ is $\sigma$-stable and  $H$ is a nonempty subset of $G$, then $\sigma(H):=\{\sum_{n=1}^{\infty}I_{a_{n}}h_{n}: \{h_{n},n\in \mathbb{N}\}$ is a sequence in $H$ and $\{a_{n},n\in \mathbb{N}\}\in p(1)$\} is called  the $\sigma$-stable hull of $H$.
    \end{definition}

   \par 
   It is clear that $L^{0}(\mathcal{F})$ is $\sigma$-stable. In particular, $L^{0}(\mathcal{F},\mathbb{N})$ is a $\sigma$-stable directed set as a subset of $(L^{0}(\mathcal{F}),\leq)$.

   \par 
   Now, the connection between some basic results derived from the two topologies $\mathcal{T}_{\varepsilon,\lambda}$ and $\mathcal{T}_{c}$ can be summarized in Proposition \ref{prop.stable} below.

    \begin{proposition}[\cite{Guo2010}]\label{prop.stable}
    Let $(E,\|\cdot\|)$ be an $RN$ module over $\mathbb{K}$ with base $(\Omega,\mathcal{F},P)$ and $G$ be a nonempty subset of $E$. Then we have the following:
    \begin{enumerate}[(1)]
    \item If $G$ is $\sigma$-stable, then $G_{\varepsilon,\lambda}^{-}=G_{c}^{-}$, where $G_{\varepsilon,\lambda}^{-}$ and $G_{c}^{-}$  denote the closure of $G$ under $\mathcal{T}_{\varepsilon,\lambda}$ and $\mathcal{T}_{c}$, respectively.
    \item If $G$ is finitely stable, then $G$ is $\mathcal{T}_{\varepsilon,\lambda}$-complete iff $G$ is both $\sigma$-stable and $\mathcal{T}_{c}$-complete.  
    \end{enumerate}	
    \end{proposition}

   \begin{remark}\label{rmk.same}
   	By Proposition \ref{prop.stable}, a $\sigma$-stable subset $G$ of an $RN$ module has the same closedness (resp., completeness) under  $\mathcal{T}_{\varepsilon,\lambda}$ and $\mathcal{T}_{c}$. Hence, in the remainder of this paper, whenever $G$ is $\sigma$-stable, we will simply refer to $G$ as closed (resp., complete) without specifying the topology.
   \end{remark}

   \par 
   Motivated by the randomized version of the Bolzano–Weierstrass theorem \cite{KS2001}, Guo et al. introduced and systematically studied the notion of random sequential compactness in $RN$ modules in \cite{GWXYC2025}.

   \begin{definition}[\cite{GWXYC2025}]\label{defn.random sequentially compact}
   	Let $(E, \|\cdot\|)$ be an $RN$ module over $\mathbb{K}$ with base $(\Omega, \mathcal{F}, P)$ and $G$ a nonempty subset such that $G$ is contained in a $\sigma$-stable subset $H$ of $E$. Given a sequence $\{x_{n}, n\in \mathbb{N} \}$ in $G$, a sequence $\{y_{k}, k\in \mathbb{N} \}$ in $H$ is called a random subsequence of $\{x_{n}, n\in \mathbb{N} \}$ if there exists a sequence $\{n_{k}, k\in \mathbb{N} \}$ in $L^{0}(\mathcal{F},\mathbb{N})$ such that the following two conditions are satisfied:
   	\begin{enumerate} [(1)]
   		\item $n_{k}<n_{k+1}$ on $1$ for any $k\in \mathbb{N}$;
   		\item $y_{k}=x_{n_{k}}:= \sum^{\infty}_{l=1} I_{(n_{k}=l)} x_{l}$ for each $k\in \mathbb{N}$.
   	\end{enumerate}
   	Further, $G$ is said to be random sequentially compact if there exists a random subsequence $\{y_{k}, k\in \mathbb{N}\}$ of $\{x_{n}, n\in \mathbb{N} \}$ for any sequence $\{x_{n}, n\in \mathbb{N} \}$ in $G$ such that $\{y_{k}, k\in \mathbb{N}\}$ converges in $\mathcal{T}_{\varepsilon,\lambda}$ to some element in $G$.
   \end{definition}

   \par 
   A careful reader will find that we require $n_{k}$ to be a positive integer-valued measurable function in the definition of a random subsequence in \cite{GWXYC2025} instead of an element in $L^0(\mathcal{F},\mathbb{N})$ in Definition \ref{defn.random sequentially compact}, but it is easy to check that the two formulations are essentially equivalent!

   \par 
   Let $X$ and $Y$ be two nonempty sets and $F:X\rightarrow 2^{Y}\setminus \{\emptyset\}$ be a set-valued mapping. For any nonempty sets $G\subseteq X$ and $M\subseteq Y$, the image of $G$ under $F$ is defined by 
   $$F(G)=\bigcup_{x\in G}F(x),$$
   and the upper inverse of $M$ is defined by
   $$F^{+}(M)=\{x\in X:F(x)\subseteq M\}.$$

    \par 
    Let $G$ be a $\sigma$-stable subset of an $L^{0}(\mathcal{F},\mathbb{K})$-module. For any sequence of nonempty subsets $\{G_{n}, n\in \mathbb{N}\}$ of $G$ and any $\{a_{n}, n \in\mathbb{N}\}\in p(1)$, 
    $$\sum^{\infty}_{n=1} I_{a_{n}} G_{n}:=\{\sum^{\infty}_{n=1}I_{a_{n}} x_{n}: x_{n}\in G_{n}, \forall n\in\mathbb{N}\}$$ 
    is called the countable concatenation of $\{G_{n},n\in\mathbb{N}\}$ along $\{a_n,n\in \mathbb{N}\}$.

    \begin{definition}\label{defn.rusc}
    	Let $(E_{1},\|\cdot\|_{1})$ and $(E_{2},\|\cdot\|_{2})$ be two $RN$ modules over $\mathbb{K}$ with base $(\Omega,\mathcal{F},P)$,
    	$X\subseteq E_{1}, Y\subseteq E_{2}$  two nonempty sets and  $F:X\rightarrow 2^{Y}\setminus \{\emptyset\}$ a set-valued mapping.
    	$F$ is said to be  
    	\begin{enumerate}[(1)]
    		\item $\sigma$-stable if both $X$ and $F(X)$ are $\sigma$-stable and 
    		$$F(\sum_{n=1}^{\infty}I_{a_{n}}x_{n})=\sum_{n=1}^{\infty}I_{a_{n}}F(x_{n})$$
    		for any sequence $\{x_{n},n\in \mathbb{N}\}$ in $X$ and any $\{a_{n},n\in \mathbb{N}\}\in p(1)$.
    		\item $\mathcal{T}_{c}$-upper semicontinuous at $x_{0}\in X$ if for every $\mathcal{T}_{c}$-neighborhood $U$ of $F(x_{0})$, there is a $\mathcal{T}_{c}$-neighborhood $V$ of $x_{0}$ such that  $F(V\cap X)\subseteq U$ (equivalently, the upper inverse $F^{+}(U)$ is a $\mathcal{T}_{c}$-neighborhood of $x_{0}$ in $X$). Furthermore,  $F$ is said to be $\mathcal{T}_{c}$-upper semicontinuous on $X$ if it is $\mathcal{T}_{c}$-upper semicontinuous at every point $x\in X$.
    	\end{enumerate}
    \end{definition}

    \par 
    A point $x_{0}\in X$ is said to be a fixed point of the set-valued mapping $F:X\rightarrow 2^{X}$ if $x_{0}\in F(x_{0})$.

    \par 
    Now we can give the main result of this paper.

    \begin{theorem}\label{thm.Kakutani}
    Let $(E,\|\cdot\|)$ be an $RN$ module over $\mathbb{K}$ with base $(\Omega,\mathcal{F},P)$, $G$ a random sequentially compact $L^{0}$-convex subset of $E$ and $F:G\rightarrow 2^{G}\backslash\{\emptyset\}$ a $\sigma$-stable $\mathcal{T}_{c}$-upper semicontinuous mapping such that $F(x)$ is closed and $L^{0}$-convex for each $x\in G$. Then $F$ has a fixed point.
    \end{theorem}

    \begin{remark}
    \begin{enumerate}[(1)]
    	\item Since $G$ is random sequentially compact and $L^{0}$-convex, by part (6) of \cite[Lemma 3.5]{GWXYC2025} and part (2) of Proposition \ref{prop.stable} $G$ is $\sigma$-stable and $\mathcal{T}_{\varepsilon,\lambda}$-complete, which ensures that the $\sigma$-stability of $F$ is well defined. Moreover, since $F$ is $\sigma$-stable, it is easy to verify that $F(x)$ is $\sigma$-stable for any $x \in X$. Consequently, for any $x \in X$, the closedness of $F(x)$ is understood as in Remark \ref{rmk.same}.
    	\item If $(\Omega,\mathcal{F},P)$ is trivial, namely, $\mathcal{F}=\{\emptyset,\Omega\}$, the $RN$ module $(E,\|\cdot\|)$ reduces to an ordinary normed space, $G$ to a compact convex subset of $E$ and $F$ to an ordinary upper semicontinuous set-valued mapping. Hence, Theorem \ref{thm.Kakutani} generalizes the classical Kakutani fixed point theorem \cite{BK1950,Kakutani1941,Nikaid1968}. 
    	\item When $F$ is single-valued, it reduces to a $\sigma$-stable $\mathcal{T}_{c}$-continuous mapping.  Consequently, by \cite[Lemma 4.3]{GWXYC2025}, Theorem \ref{thm.Kakutani} also generalizes the noncompact Schauder fixed point theorem \cite[Theorem 2.12]{GWXYC2025}.
    \end{enumerate}	
    \end{remark}

    \par 
    We conclude the section by giving an improved version of \cite[Lemma 4.4]{GWXYC2025} (namely, Proposition \ref{prop.converge a.s.3} below), where we impose the additional assumption that the related random sequentially continuous mapping $f$ is $\sigma$-stable. Random Brouwer fixed point theorem \cite[Lemma 4.5 or Lemma 4.6]{GWXYC2025} is a basis for a noncompact Schauder fixed point theorem \cite[Theorem 2.12]{GWXYC2025}, and \cite[Lemma 4.4]{GWXYC2025} plays an essential role in the proof of \cite[Lemma 4.5]{GWXYC2025}. The improved version shows that the random Brouwer fixed point theorem and the noncompact Schauder fixed point theorem established in \cite{GWXYC2025} are both correct since the random sequentially continuous mappings involved in the two theorems are both $\sigma$-stable. Besides, we would like to suggest that the reader refer to our recent work \cite{TMGC2025} for a new complete proof of the random Brouwer fixed point theorem.

    \par 
    Besides the $\sigma$-stability of the random sequentially continuous mapping in \cite[Lemma 4.4]{GWXYC2025} was not assumed, the original proof of \cite[Lemma 4.4]{GWXYC2025} used part (2) of \cite[Lemma 3.5]{GWXYC2025}. Part (2) of \cite[Lemma 3.5]{GWXYC2025} said that, in an $RN$  module $(E,\|\cdot\|)$, if a sequence $\{x_{n},n\in \mathbb{N}\}$ in some $\sigma$-stable subset of $E$ converges in $\mathcal{T}_{\varepsilon,\lambda}$ to $x_{0}\in E$, then any random subsequence $\{x_{n_{k}},k\in \mathbb{N}\}$ of $\{x_{n},n\in \mathbb{N}\}$ converges in $\mathcal{T}_{\varepsilon,\lambda}$ to $x_{0}$. Unfortunately, there exist some counterexamples showing that part (2) of \cite[Lemma 3.5]{GWXYC2025} does not necessarily hold. Fortunately, Lemma \ref{lemm.converge a.s.1} below can be used to give a new proof of the improved version of \cite[Lemma 4.4]{GWXYC2025}.

    \begin{lemma}\label{lemm.converge a.s.1}
    	Let $(E,\|\cdot\|)$ be an $RN$ module over $\mathbb{K}$ with base $(\Omega, \mathcal{F}, P)$ and $\{x_{n},n\in \mathbb{N}\}$ be a sequence in some $\sigma$-stable subset of $E$ such that $\{x_{n},n\in \mathbb{N}\}$ converges a.s. to $x_{0}\in E$, namely, $\{ \|x_{n}-x_{0}\| ,n\in \mathbb{N} \}$ converges a.s. to 0. Then any random subsequence 
    	$\{x_{n_k},k\in \mathbb{N}\}$ of $\{x_{n},n\in \mathbb{N}\}$ converges a.s. to $x_{0}$. 
    \end{lemma}
    \begin{proof}
    	Let $\varepsilon>0$ be a given real number. Since $n_{k}<n_{k+1}$ on $1$ for any $k\in \mathbb{N}$ implies that $n_{k}\geq k$ for each $k\in \mathbb{K}$, then we have 
    	\begin{align}
    		(\|x_{n_{k}}-x\|\geq \varepsilon)&=(\sum_{l=1}^{\infty}I_{(n_{k}=l)}\|x_{l}-x\|\geq \varepsilon)\nonumber\\
    		&=\bigvee_{l=1}^\infty [(n_{k}=l)\wedge (\|x_{l}-x\|\geq \varepsilon)]\nonumber\\
    		&=\bigvee_{l=k}^\infty [(n_{k}=l)\wedge (\|x_{l}-x\|\geq \varepsilon)]\nonumber\\
    		&\leq \bigvee_{l=k}^\infty (\|x_{l}-x\|\geq \varepsilon)\nonumber
    	\end{align}
    	for each $k\in \mathbb{N}$. Furthermore, since $\{x_{n},n\in \mathbb{N}\}$ converges a.s. to $x_{0}\in E$, then, in the language of measure algebra, we have 
    	$$\bigwedge_{m=1}^{\infty}\bigvee_{l=m}^{\infty}(\|x_{l}-x_{0}\|\geq \varepsilon)=0.$$
        It follows that 
    	\begin{align}
    		\bigwedge_{m=1}^{\infty}\bigvee_{k=m}^{\infty}(\|x_{n_{k}}-x\|\geq \varepsilon)&\leq \bigwedge_{m=1}^{\infty}\bigvee_{k=m}^{\infty}\bigvee_{l=k}^\infty (\|x_{l}-x\|\geq \varepsilon)\nonumber\\
    		&=\bigwedge_{m=1}^{\infty}\bigvee_{l=m}^\infty (\|x_{l}-x\|\geq \varepsilon)\nonumber\\
    		&=0,\nonumber
    	\end{align}
    	implying $\{\|x_{n_{k}}-x\|, k\in \mathbb{N}\}$ converges a.s. to $0$. Thus, $\{x_{n_{k}},k\in \mathbb{N}\}$ converges a.s. to $x_{0}$.
    \end{proof}

    \begin{definition}[\cite{GWXYC2025}]
    	Let $(E_{1}, \|\cdot\|_{1})$ and $(E_{2}, \|\cdot\|_{2})$ be two $RN$ modules over $\mathbb{K}$ with base $(\Omega, \mathcal{F}, P)$, $G_{1}$ and $G_{2}$ two nonempty subsets of $E_{1}$ and $E_{2}$, respectively, and $f$ a mapping from $G_{1}$ to $G_{2}$. $f$ is said to be:
    	\begin{enumerate}[(1)]
    		\item $\sigma$-stable if both $G_{1}$ and $G_{2}$ are $\sigma$-stable and 
    		$$f(\sum^{\infty}_{k=1}I_{a_{k}}x_{k})=\sum^{\infty}_{k=1}I_{a_{k}}f(x_{k})$$ for any $\{a_{k}, k\in\mathbb{N}\}\in p(1)$ and any sequence $\{x_{n}, n\in\mathbb{N}\}$ in $G_{1}$.
    		\item random sequentially continuous at $x_{0}\in G_{1}$ if $G_{1}$ is $\sigma$-stable and if for any sequence $\{x_{n}, n\in\mathbb{N}\}$ in $G_{1}$ convergent in $\mathcal{T}_{\varepsilon,\lambda}$ to $x_{0}$ there exists a random subsequence $\{x_{n_{k}}, k\in\mathbb{N}\}$ of $\{x_{n}, n\in\mathbb{N} \}$ such that $\{f(x_{n_{k}}), k\in \mathbb{N}\}$ converges in $\mathcal{T}_{\varepsilon,\lambda}$ to $f(x_{0})$. Further, $f$ is said to be random sequentially continuous if $f$ is random sequentially continuous at any point in $G_{1}$.
    	\end{enumerate}
    \end{definition}

    \begin{lemma}\label{lemm.converge a.s.2}
    	Let $(E_{1}, \|\cdot\|_{1})$ and $(E_{2}, \|\cdot\|_{2})$ be two $RN$ modules over $\mathbb{K}$ with base $(\Omega, \mathcal{F}, P)$, $G_{1}\subseteq E_{1}$ and $G_{2}\subseteq E_{2}$ two nonempty $\sigma$-stable subsets, and $f:G_{1}\rightarrow G_{2}$ a random sequentially continuous mapping. Then for any sequence $\{x_{n},n\in \mathbb{N}\}$ in $G_{1}$ that converges in $\mathcal{T}_{\varepsilon,\lambda}$ to $x_{0}\in G_{1}$, there exists a random subsequence $\{x_{n_{k}},k\in \mathbb{N}\}$ of $\{x_{n},n\in \mathbb{N}\}$ such that $\{x_{n_{k}},k\in \mathbb{N}\}$ converges a.s. to $x_{0}$ and $\{f(x_{n_{k}}),k\in \mathbb{N}\}$ converges a.s. to $f(x_{0})$. 
    \end{lemma}
    \begin{proof}
        Since $\{x_{n},n\in \mathbb{N}\}$ converges in $\mathcal{T}_{\varepsilon,\lambda}$ to $x_{0}\in G$, there exists a subsequence $\{x'_{i},i\in \mathbb{N}\}$ of $\{x_{n},n\in \mathbb{N}\}$ such that  $\{x'_{i},i\in \mathbb{N}\}$ converges a.s. to $x_{0}$. We can assume, without loss of generality, that $\{x'_{i},i\in \mathbb{N}\}$ is just $\{x_{n},n\in \mathbb{N}\}$ itself, namely, $\{x_{n},n\in \mathbb{N}\}$ converges a.s. to $x_{0}$. Further, since $f$ is random sequentially continuous, there exists a  random subsequence $\{x_{n_{k}},k\in \mathbb{N}\}$ of $\{x_{n},n\in \mathbb{N}\}$ such that $\{f(x_{n_{k}}),k\in \mathbb{N}\}$ converges in $\mathcal{T}_{\varepsilon,\lambda}$ to $f(x_{0})$.

        \par 
        For $\{f(x_{n_{k}}),k\in\mathbb{N}\}$, there exists a subsequence $\{f(x_{n_{k_{l}}}),l\in\mathbb{N}\}$ that converges a.s. to $f(x_{0})$. Clearly,  $\{x_{n_{k_{l}}},l\in\mathbb{N}\}$ is a subsequence of $\{x_{n_{k}},k\in\mathbb{N}\}$ and hence a random subsequence of $\{x_{n},n\in\mathbb{N}\}$. By Lemma \ref{lemm.converge a.s.1}, $\{x_{n_{k_{l}}},l\in\mathbb{N}\}$ also converges a.s. to $x_{0}$.
    \end{proof}

    \par
    The proof of Proposition \ref{prop.converge a.s.3} below is merely a slight modification to the original proof of \cite[Lemma 4.4]{GWXYC2025}. The reader will find that the new proof essentially only replace part (2) of \cite[Lemma 3.5]{GWXYC2025} with Lemma \ref{lemm.converge a.s.1} in the original proof of \cite[Lemma 4.4]{GWXYC2025}.

    \begin{proposition}\label{prop.converge a.s.3}
    	Let $(E_{1},\|\cdot\|)$ and $(E_{2},\|\cdot\|)$ be two $RN$ modules over $\mathbb{K}$ with base $(\Omega,\mathcal{F},P)$, $G_{1}\subset E_{1}$ and $G_{2}\subset E_{2}$ two $\sigma$-stable subsets, and $f: G_{1}\rightarrow G_{2}$ a $\sigma$-stable random sequentially continuous mapping. For a sequence $\{(x_{1}^{m}, x_{2}^{m},\cdots, x_{l}^{m}), m\in \mathbb{N} \}$ in $G_{1}^{l}$, where $l$ is a fixed positive integer and $G^{l}_{1}$ is the $l$--th Cartesian power set of $G_{1}$, if there exists a random subsequence $\{(x_{1}^{M_{n}^{(0)}}, x_{2}^{M_{n}^{(0)}}, \cdots, x_{l}^{M_{n}^{(0)}}), n\in \mathbb{N}\}$ of which such that $\{x_{i}^{M_{n}^{(0)}}, n\in \mathbb{N}\}$ converges in $\mathcal{T}_{\varepsilon,\lambda}$ to some $y_{i}\in G_{1}$ for each $i\in \{1,2,\cdots, l\}$, then there exists a random subsequence $\{(x_{1}^{M_{n}}, x_{2}^{M_{n}}, \cdots, x_{l}^{M_{n}}),n\in \mathbb{N}\}$ of $\{(x_{1}^{M_{n}^{(0)}}, x_{2}^{M_{n}^{(0)}}, \cdots, x_{l}^{M_{n}^{(0)}}), n\in \mathbb{N}\}$ such that $\{x_{i}^{M_{n}}, n\in \mathbb{N}\}$ converges in $\mathcal{T}_{\varepsilon,\lambda}$ to $y_i$ and $\{f(x_{i}^{M_{n}}),n\in \mathbb{N}\}$ converges in $\mathcal{T}_{\varepsilon,\lambda}$ to $f(y_{i})$ for each $i\in \{1,2,\cdots, l\}$.
    \end{proposition}
    \begin{proof}
    	Since $\{x_{i}^{M_{n}^{(0)}}, n\in \mathbb{N}\}$ converges in $\mathcal{T}_{\varepsilon,\lambda}$ to $y_{i}$ for each $i\in \{1,2,\cdots,l\}$, there exists a subsequence $\{(x_{1}^{N_{n}}, x_{2}^{N_{n}},\cdots, x_{l}^{N_{n}}),n\in \mathbb{N}\}$ of $\{(x_{1}^{M_{n}^{(0)}}, x_{2}^{M_{n}^{(0)}},$ $ \cdots, x_{l}^{M_{n}^{(0)}}), n\in \mathbb{N}\}$ such that $\{x_{i}^{N_{n}},n\in \mathbb{N}\}$ converges a.s. to $y_{i}$ for each $i\in \{1,2,\cdots,l\}$. We can assume, without loss of generality, that $\{(x_{1}^{N_{n}}, x_{2}^{N_{n}},$ $\cdots, x_{l}^{N_{n}}),n\in \mathbb{N}\}$ is just $\{(x_{1}^{M_{n}^{(0)}}, x_{2}^{M_{n}^{(0)}}, \cdots, x_{l}^{M_{n}^{(0)}}), n\in \mathbb{N}\}$ itself, namely $\{x_{i}^{M_{n}^{(0)}}, n\in \mathbb{N}\}$ converges a.s. to $y_{i}$ for each $i\in \{1,2,\cdots,l\}$.

    	\par 
    	For $\{x_{1}^{M_{n}^{(0)}}, n\in \mathbb{N}\}$, since $f$ is random sequentially continuous, by Lemma \ref{lemm.converge a.s.2} there exists a random subsequence $\{x_{1}^{M^{(0)}_{n_{k}}}, k\in \mathbb{N}\}$ of $\{x_{1}^{M_{n}^{(0)}}, n\in \mathbb{N}\}$ such that $\{x_{1}^{M^{(0)}_{n_{k}}}, k\in\mathbb{N}\}$ converges a.s. to $y_{1}$ and $\{f(x_{1}^{M^{(0)}_{n_{k}}}), k\in\mathbb{N}\}$ converges a.s. to $f(y_{1})$. Let $M^{(1)}_{k}= M^{(0)}_{n_{k}}$ for each $k\in\mathbb{N}$, namely,  $M^{(1)}_{k}=\sum_{l=1}^{\infty}I_{(n_{k}=l)}M^{(0)}_{l}$. Then  $\{(x_{1}^{M_{n}^{(1)}},x_{2}^{M_{n}^{(1)}},\cdots,x_{l}^{M_{n}^{(1)}}), n\in \mathbb{N}\}$ is a random subsequence of $\{(x_{1}^{M_{n}^{(0)}}, x_{2}^{M_{n}^{(0)}},$ $ \cdots, x_{l}^{M_{n}^{(0)}}), n\in \mathbb{N}\}$ such that $\{x_{i}^{M_{n}^{(1)}}, n\in\mathbb{N}\}$ still converges a.s. to $y_{i}$ for each $i\in \{1,2,\cdots,l\}$ (by Lemma \ref{lemm.converge a.s.1}) and $\{f(x_{1}^{M_{n}^{(1)}}), n\in\mathbb{N}\}$ converges a.s. to $f(y_{1})$.

    	\par 
    	For $\{x_{2}^{M_{n}^{(1)}}, n\in \mathbb{N}\}$, by Lemma \ref{lemm.converge a.s.2} there exists a random subsequence $\{x_{2}^{M_{n_{k}}^{(1)}}, k\in \mathbb{N}\}$ of $\{x_{2}^{M_{n}^{(1)}},n\in \mathbb{N}\}$ such that $\{x_{2}^{M_{n_{k}}^{(1)}}, k\in \mathbb{N}\}$ converges a.s. to $y_{2}$ and $\{f(x_{2}^{M_{n_{k}}^{(1)}}), k\in \mathbb{N}\}$ converges a.s. to $f(y_{2})$. Since $f$ is $\sigma$-stable, it is easy to see that $\{f(x_{1}^{M_{n_{k}}^{(1)}}), k\in \mathbb{N}\}$ is also a random subsequence of $\{f(x_{1}^{M_{n}^{(1)}}), n\in \mathbb{N}\}$, then $\{f(x_{1}^{M_{n_{k}}^{(1)}}), k\in \mathbb{N}\}$ still converges a.s. to $f(y_{1})$. Let $M^{(2)}_{k}= M^{(1)}_{n_{k}}$ for each $k\in\mathbb{N}$, then $\{(x_{1}^{M_{n}^{(2)}},x_{2}^{M_{n}^{(2)}},\cdots,x_{l}^{M_{n}^{(2)}}),n\in \mathbb{N}\}$ is a random subsequence of $\{(x_{1}^{M_{n}^{(0)}}, x_{2}^{M_{n}^{(0)}}, \cdots, x_{l}^{M_{n}^{(0)}}), n\in \mathbb{N}\}$ such that $\{x_{i}^{M_{n}^{(2)}},n\in\mathbb{N}\}$ still converges a.s. to $y_{i}$ for each $i\in \{1,2,\cdots, l\}$ and $\{f(x_{i}^{M_{n}^{(2)}}), n\in \mathbb{N}\}$ converges a.s. to $f(y_{i})$ for each $i\in \{1,2\}$.

    	 \par 
    	 Inductively, we can obtain a random subsequence $\{(x_{1}^{M_{n}^{(l)}}, x_{2}^{M_{n}^{(l)}},\cdots, x_{l}^{M_{n}^{(l)}}),$ $n\in \mathbb{N}\}$ of $\{(x_{1}^{M_{n}^{(0)}}, x_{2}^{M_{n}^{(0)}}, \cdots, x_{l}^{M_{n}^{(0)}}), n\in \mathbb{N}\}$ such that $\{x_{i}^{M_{n}^{(l)}},n\in \mathbb{N}\}$ converges a.s. to $y_{i}$ for each $i\in \{1,2,\cdots,l\}$ and $\{f(x_{i}^{M_{n}^{(l)}}),n\in \mathbb{N}\}$ converges a.s. to $f(y_{i})$ for each $i\in \{1,2,\cdots,l\}$. Finally, taking $M_{n}=M_{n}^{l}$ for each $n\in \mathbb{N}$, we can, of course, obtain our desired result.
    \end{proof}

%################################################################################################
%################################################################################################
\section{Proof of Theorem \ref{thm.Kakutani}: Random Kakutani fixed point theorem}\label{sec.3}
%################################################################################################
%################################################################################################

\par  
As pointed out in Section \ref{sec.1}, the proof of Theorem \ref{thm.Kakutani} relies on the equivalence between a $\sigma$-stable random sequentially compact set and a stably sequentially compact set. To present the notion of a stably sequentially compact set, we first recall Definition \ref{defn.stable sequence} below, where the notion of a stable subsequence is a strengthened version of the original notion introduced in \cite{GMT2024}.

\begin{definition}\label{defn.stable sequence}
	Let $(E,\|\cdot\|)$ be an $RN$ module over $\mathbb{K}$ with base $(\Omega,\mathcal{F},P)$, $G$ a $\sigma$-stable subset of $E$ and $\{x_{n},n\in \mathbb{N}\}$ a sequence in $G$. 
	\begin{enumerate}[(1)]
		\item A $\sigma$-stable mapping $x$ from $L^{0}(\mathcal{F},\mathbb{N})$ to $G$ is called a stable sequence in $G$, denoting $x(u)$ by $ x_{u}$ for each $u\in L^{0}(\mathcal{F},\mathbb{N})$. 
		\item For each $u\in L^0(\mathcal{F},\mathbb{N})$, define 
		$$x_{u}=\sum_{n=1}^{\infty}I_{(u=n)}x_{n}.$$
		Then $\{x_{u},u\in L^0(\mathcal{F},\mathbb{N})\}$ is a stable sequence in $G$, called the stable sequence generated by $\{x_{n},n\in \mathbb{N}\}$.  
		\item A stable sequence $\{y_{v},v\in L^0(\mathcal{F},\mathbb{N})\}$ is called a stable subsequence of a stable sequence $\{x_{u},{u}\in L^0(\mathcal{F},\mathbb{N})\}$ if there exists a $\sigma$-stable mapping $\varphi:L^0(\mathcal{F},\mathbb{N})\rightarrow L^0(\mathcal{F},\mathbb{N})$ such that the following two conditions are satisfied:
		\begin{enumerate}
			\item [(i)] $y_{v}=x_{\varphi(v)}$ for each $v\in L^0(\mathcal{F},\mathbb{N})$;
			\item [(ii)] $\varphi(n)<\varphi(n+1)$ on $1$ for each $n\in \mathbb{N}$.
		\end{enumerate}
	  For simplicity, we may also write $\{x_{\varphi(v)}, v\in L^0(\mathcal{F},\mathbb{N})\}$ for $\{y_{v}, v\in L^0(\mathcal{F},\mathbb{N})\}$.
	\end{enumerate}
\end{definition}

\par 
According to part (3) of Definition \ref{defn.stable sequence}, it is easy to check that $\{x_{\varphi(v)}, v\in L^0(\mathcal{F},\mathbb{N})\}$ is also a subnet of $\{x_{u}, u\in L^0(\mathcal{F},\mathbb{N})\}$.

\par 
Let $E$ be an $L^{0}(\mathcal{F},\mathbb{K})$-module and $G$ be a $\sigma$-stable subset of $E$. A subset $H$ of $G$ is said to be stably finite if there exist a sequence $\{G_{n},n\in \mathbb{N}\}$ of nonempty finite subsets of $G$ and $\{a_n,n\in \mathbb{N}\}\in p(1)$ such that $H=\sum_{n=1}^{\infty}I_{a_{n}}\sigma(G_{n})$.

\begin{definition}[\cite{GMT2024}]\label{defn.stably sequentically compactness}
	Let $(E,\|\cdot\|)$ be an $RN$ module over $\mathbb{K}$ with base $(\Omega,\mathcal{F},P)$ and $G$ be a $\sigma$-stable subset of $E$. $G$ is said to be 
	\begin{enumerate} [(1)]
		\item stably sequentially compact if every stable sequence in $G$ admits a stable subsequence that converges in $\mathcal{T}_{c}$ to a point of $G$.
		\item random totally bounded if, for any given $\varepsilon\in L^{0}_{++}(\mathcal{F})$, there exist a sequence $\{G_{n},n\in \mathbb{N}\}$ of nonempty finite subsets of $G$ and $\{a_n,n\in \mathbb{N}\}\in p(1)$ such that 
		$$I_{a_{n}}G\subseteq I_{a_n}[\sigma(G_{n})+B(\theta,\varepsilon)] ~\text{for~each}~ n\in \mathbb{N}.$$  
	\end{enumerate}
\end{definition}

\begin{remark}\label{rmk.net}
In part (2) of Definition \ref{defn.stably sequentically compactness}, $I_{a_{n}}G\subseteq I_{a_n}[\sigma(G_{n})+B(\theta,\varepsilon)]$ for each $n\in \mathbb{N}$ implies that 
$$G\subseteq \sum_{n=1}^{\infty}I_{a_{n}}\sigma(G_{n})+B(\theta,\varepsilon)=\bigcup_{x\in \sum_{n=1}^{\infty}I_{a_{n}}\sigma(G_{n})}B(x,\varepsilon).$$	
Hence, every random totally bounded set necessarily possesses a stably finite random $\varepsilon$-net. It should be noted that a stably finite set is neither finite nor even countable in general, and therefore a random totally bounded set is much more complicated than a classical totally bounded set.
\end{remark}

\par 
One can easily see that the strengthened notion of a stable subsequence still makes Proposition \ref{prop.stably compact} below hold, which will play a crucial role in the proof of Theorem \ref{thm.Kakutani}.

\begin{proposition}[\cite{GMT2024}]\label{prop.stably compact}
	Let $(E,\|\cdot\|)$ be an $RN$ module over $\mathbb{K}$ with base $(\Omega,\mathcal{F},P)$ and $G$ be a $\sigma$-stable subset of $E$. Then the following statements are equivalent:
	\begin{enumerate}[(1)]
		\item [(1)] $G$ is stably sequentially compact.
		\item [(2)]	$G$ is random totally bounded and complete.
		\item [(3)]	$G$ is random sequentially compact.	
	\end{enumerate}
\end{proposition}

\begin{remark}\label{lemm.special case}
	In \cite{GMT2024}, the notions of stably sequentially compact sets and random totally bounded sets were introduced in a $d$-$\sigma$-stable random metric space. Definition \ref{defn.stably sequentically compactness} and Proposition \ref{prop.stably compact} are in fact special cases of \cite[Definition 2.19]{GMT2024} and \cite[Theorem 2.12]{GMT2024}, respectively, since a $\sigma$-stable subset $G$ of an $RN$ module $(E,\|\cdot\|)$ naturally forms a $d$-$\sigma$-stable random metric space $(G,d)$, where the random metric $d:G\times G\rightarrow L^{0}_{+}(\mathcal{F})$ is defined by $d(x,y)=\|x-y\|$ for any $x,y\in G$ (see \cite{GMT2024} for details). Here, we present only these special cases, since the present work is only concerned with the setting of $RN$ modules.
\end{remark}

\begin{lemma}\label{lemm.basic}
	Let $(E_{1},\|\cdot\|_{1})$ and $(E_{2},\|\cdot\|_{2})$ be two $RN$ modules over $\mathbb{K}$ with base $(\Omega,\mathcal{F},P)$, $X\subseteq E_{1}, Y\subseteq E_{2}$ two $\sigma$-stable subsets and $F:X\rightarrow 2^{Y}\setminus \{\emptyset\}$ a $\sigma$-stable mapping. For any $a\in B_{\mathcal{F}}$ and any $x\in E_{1}$, if there exists a finitely stable subset $G$ of $X$ such that $I_{a}x\in I_{a}G$, then $I_{a}F(x)\subseteq I_{a}F(G)$. 
\end{lemma}
\begin{proof}
    Arbitrarily choose $z\in G$ and let $x_{1}=I_{a}x+I_{a^{c}}z$. Then $x_{1}\in G$, and we have 
	$$I_{a}F(x)+I_{a^{c}}F(z)=F(I_{a}x+I_{a^{c}}z)=F(x_{1})\subseteq F(G),$$
	implying $I_{a}F(x)\subseteq I_{a}F(G)$.
\end{proof}

 \par 
 Following the methodology employed in the proof of \cite[Lemma 4.8]{GWXYC2025}, we can establish Lemma \ref{lemm.projection} below. For completeness, we provide a detailed proof here.

 \begin{lemma}\label{lemm.projection}
 Let $(E,\|\cdot\|)$ be an $RN$ module over $\mathbb{K}$ with base $(\Omega,\mathcal{F},P)$ and $G$ be a $\sigma$-stable $L^{0}$-convex subset of $E$. For $r\in L^{0}_{++}(\mathcal{F})$ and $a\in B_{\mathcal{F}}$ with $a>0$, suppose that there exists a finite subset $\{x_{1},\cdots,x_{k}\}$ of $G$ such that $I_{a}G\subseteq I_{a}[\sigma(\{x_{1},\cdots,x_{k}\})+B(\theta,r)]$. For a finite subset $\{y_{1},\cdots,y_{k}\}$ of $G$, define a mapping $g:G\rightarrow G$ by 
 $$g(x)=\dfrac{1}{\sum_{j=1}^{k}\alpha_{j}(x)}\sum_{i=1}^{k}\alpha_{i}(x)y_{i}, \forall x\in G,$$
 where for each $i\in \{1,\cdots,k\}$, the mapping $\alpha_{i}:G\rightarrow L^{0}_{+}(\mathcal{F})$ is defined by 
 $$\alpha_{i}(x)=\max\{0,r-\|I_{a}x-I_{a} x_{i}\|\},~\forall x\in G.$$
 Then we have the following statements:
 \begin{enumerate}[(1)]
 	\item $\sum_{i=1}^{k}\alpha_{i}(x)>0$ on $1$ for any $x\in G$.
 	\item $g$ is well defined, $\sigma$-stable and $\mathcal{T}_{c}$-continuous.
 \end{enumerate}	
 \end{lemma}
 \begin{proof}
 	(1) For any given $x\in G$, since 
 	$$I_{a}G\subseteq I_{a}[\sigma(\{x_{1},\cdots,x_{k}\})+B(\theta,r)]=\sigma(\{I_{a}x_{1},\cdots,I_{a}x_{k}\})+I_{a}B(\theta,r),$$
 	there exists $\{b_{i},i=1\sim k\}\in p(1)$ such that 
 	$$I_{a}x\in \sum_{i=1}^{k}I_{a\wedge b_{i}}x_{i}+I_{a}B(\theta,r).$$
 	Hence, for each $i\in \{1,\cdots,k\}$,
 	$$I_{a\wedge b_{i}}x\in I_{a\wedge b_{i}}x_{i}+I_{a\wedge b_{i}}B(\theta,r),$$
 	implying 
 	$$r-\|I_{a}x-I_{a}x_{i}\|>0~\text{on}~a\wedge b_{i},$$
 	namely, 
 	$$\alpha_{i}(x)>0~\text{on}~a\wedge b_{i}.$$
 	Since $\alpha_{i}(x)\in L^{0}_{+}(\mathcal{F})$ for each $i\in \{1,\cdots,k\}$, it follows that $\sum_{i=1}^{k}\alpha_{i}(x)>0$ on $a$.

 	\par 
 	Besides, it is clear that $\sum_{i=1}^{k}\alpha_{i}(x)>0$ on $a^{c}$ for any $x\in G$. 
 	
 	\par 
 	To sum up, $\sum_{i=1}^{k}\alpha_{i}(x)>0$ on $1$.

 	\par 
 	(2) By (1), $g$ is well defined. For any $\{a_{n},n\in \mathbb{N}\}\in p(1)$ and any sequence $\{x_{n},n\in \mathbb{N}\}$ in $G$, since $G$ is $\sigma$-stable and each $\alpha_{i}$ is also $\sigma$-stable, we have 
 	\begin{align}
 		g(\sum_{n=1}^{\infty}I_{a_{n}}x_{n})&=\dfrac{1}{\sum_{j=1}^{k}\alpha_{j}(\sum_{n=1}^{\infty}I_{a_{n}}x_{n})}\sum_{i=1}^{k}\alpha_{i}(\sum_{n=1}^{\infty}I_{a_{n}}x_{n})y_{i}\nonumber\\
 		&=\dfrac{1}{\sum_{n=1}^{\infty}I_{a_{n}}(\sum_{j=1}^{k}\alpha_{j}(x_{n}))}\sum_{n=1}^{\infty}I_{a_{n}}(\sum_{i=1}^{k}\alpha_{i}(x_{n})y_{i})\nonumber\\
 		&=\sum_{n=1}^{\infty}I_{a_{n}}\dfrac{1}{\sum_{j=1}^{k}\alpha_{j}(x_{n})}\sum_{i=1}^{k}\alpha_{i}(x_{n})y_{i}\nonumber\\
 		&=\sum_{n=1}^{\infty}I_{a_{n}}g(x_{n}),\nonumber
 	\end{align}
 	which shows that $g$ is $\sigma$-stable.

 	\par 
 	Further, since each $\alpha_{i}$ is $\mathcal{T}_{c}$-continuous and $(E,\mathcal{T}_{c})$ is a topological module over the topological ring $(L^{0}(\mathcal{F},\mathbb{K}),\mathcal{T}_{c})$, $g$ must be $\mathcal{T}_{c}$-continuous.
 \end{proof}

 \par Now we can prove Theorem \ref{thm.Kakutani}.

\begin{proof}[Proof of Theorem \ref{thm.Kakutani}]
%	Without loss of generality, we can assume that $\theta \in G$ (otherwise, let $p_0\in G$, consider $G'=G-p_0$ and $F':G'\to 2^{G'}\backslash\{\emptyset\}$ defined by $F'(p)=F(p+p_0)-p_0$ for each $p\in G'$).
 
 Since $G$ is $\sigma$-stable and random sequentially compact, by Proposition \ref{prop.stably compact} $G$ is random totally bounded and $\mathcal{T}_{\varepsilon,\lambda}$-complete, so we can, without loss of generality, assume that $E$ is $\mathcal{T}_{\varepsilon,\lambda}$-complete (otherwise, we can consider the $\mathcal{T}_{\varepsilon,\lambda}$-completion of $E$ and note that $G$ is invariant in the process of $\mathcal{T}_{\varepsilon,\lambda}$-completion). Then $E$ is $\sigma$-stable.

	\par Fix an $n\in \mathbb{N}$. Since $G$ is random totally bounded, there exist $\{a_{m}^{n},m\in \mathbb{N}\}\in p(1)$ and a sequence $\{G_{m}^{n},m\in \mathbb{N}\}$ of finite subsets of $G$ such that
	$$G\subseteq\sum_{m=1}^{\infty}I_{a_{m}^{n}} \sigma(G_{m}^{n})+B(\theta,\frac{1}{n}).$$
	Let $G_{m}^{n}=\{x_{m,1}^{n},\cdots,x_{m,k_{m}^{n}}^{n}\}$ for any $m\in \mathbb{N}$, and further, we can, without loss of generality, assume that each $a_{m}^{n}>0$. For each $m\in\mathbb{N}$ and each $ i\in \{1,\cdots,k_{m}^{n}\}$, arbitrarily choose $y_{m,i}^{n}\in F(x_{m,i}^{n})$ and define a mapping $g_{m}^{n}:G\rightarrow G$ by 
	$$g_{m}^{n}(x)=\frac{1}{ \sum_{j=1}^{k_{m}^{n}} \alpha_{m,j}^{n} (x) } \sum_{i=1}^{k_{m}^{n}} \alpha_{m,i}^{n}(x) y_{m,i}^{n}, \forall x\in G,$$
	where $\alpha_{m,i}^{n}:G\rightarrow L^{0}_{+}(\mathcal{F})$ is defined by
	$$\alpha_{m,i}^{n}(x)=\max\{0,\frac{1}{n}-\|I_{a_{m}^{n}}x-I_{a_{m}^{n}} x_{m,i}^{n}\| \},~\forall x\in G.$$ 
	By Lemma \ref{lemm.projection}, $g_{m}^{n}$ is well defined, $\sigma$-stable and $\mathcal{T}_{c}$-continuous, and hence the noncompact Schauder fixed point theorem \cite[Theorem 2.12]{GWXYC2025} implies that $g_{m}^{n}$ has a fixed point $x_{m}^{n}\in G$, then
	$$g_{m}^{n}(x_{m}^{n})=x_{m}^{n}\in \sum_{m=1}^{\infty}I_{a_{m}^{n}} \sigma(\{x_{m,1}^{n},\cdots,x_{m,k_{m}^{n}}^{n}\})+B(\theta,\frac{1}{n}).$$
	Moreover, for each $m\in \mathbb{N}$ and each $i\in \{1,\cdots,k_{m}^{n}\}$, let
	$$\mathcal{D}_{m,i}^{n}=\{a\in B_{\mathcal{F}}:0\leq a\leq a_{m}^{n}~\text{and}~I_{a}x_{m,i}^{n}\in I_{a}(x_{m}^{n}+ B(\theta, \frac{1}{n}))\}$$
	and $d_{m,i}^{n}=\bigvee \mathcal{D}_{m,i}^{n}$, then it is easy to check that $d_{m,i}^{n}$ is attained, namely,  
	\begin{align}\label{eq.Kakutani1}
		I_{d_{m,i}^{n}}x_{m,i}^{n}\in I_{d_{m,i}^{n}}(x_{m}^{n}+ B(\theta, \frac{1}{n})),
	\end{align}
	implying
	$$\alpha_{m,i}^{n}(x_{m}^{n})>0~\text{on}~d_{m,i}^{n}~\text{and}~\alpha_{m,i}^{n}(x_{m}^{n})=0~\text{on}~a_{m}^{n}\wedge(d_{m,i}^{n})^{c}.$$
	Therefore, 
	\begin{align}\label{eq.Kakutani2}
	I_{a_{m}^{n}}\alpha_{m,i}^{n}(x_{m}^{n})=I_{d_{m,i}^{n}}\alpha_{m,i}^{n}(x_{m}^{n}), \forall m\in \mathbb{N},i\in \{1,\cdots,k_{m}^{n}\}.	
	\end{align}

    \par 
    Let $x_{n}=\sum_{m=1}^{\infty}I_{a_{m}^{n}}x_{m}^{n}$ for any $n\in \mathbb{N}$, and further let $\{x_{u},u\in L^{0}(\mathcal{F},\mathbb{N})\}$ be the stable sequence generated by $\{x_{n},n\in \mathbb{N}\}$, then by Proposition \ref{prop.stably compact} there exists a stable subsequence $\{x_{\varphi(v)}, v\in L^0(\mathcal{F},\mathbb{N})\}$ of $\{x_{u},u\in L^{0}(\mathcal{F},\mathbb{N})\}$ such that $\{x_{\varphi(v)}, v\in L^0(\mathcal{F},\mathbb{N})\}$ converges in $\mathcal{T}_{c}$ to some $x\in G$. Next we will show that $x\in F(x)$. 
	Since $F(x)$ is closed and $\{B(\theta,\frac{1}{u}):u\in L^{0}(\mathcal{F},\mathbb{N})\}$ is a $\mathcal{T}_{c}$-neighborhood base at $\theta$,  it suffices to show that
	$$x\in F(x)+B(\theta,\frac{1}{u}),~\forall u\in L^{0}(\mathcal{F},\mathbb{N}). $$

	\par For any given $u\in L^{0}(\mathcal{F},\mathbb{N})$, since $F$ is $\mathcal{T}_{c}$-upper semicontinuous, there exists $u_{1}\in L^{0}(\mathcal{F},\mathbb{N})$ such that
	\begin{equation}\label{eq.Kakutani3}
		F[G\cap (x+B(\theta, \frac{1}{u_{1}}))] \subseteq F(x)+B(\theta,\frac{1}{2u}).
	\end{equation}
	Since $\{x_{\varphi(v)}, v\in L^{0}(\mathcal{F},\mathbb{N})\}$ converges in $\mathcal{T}_{c}$ to $x$, there exists $v_{0}\in L^{0}(\mathcal{F},\mathbb{N})$ such that
	\begin{equation}\label{eq.Kakutani4}
		\varphi(v)\geq 2u_{1}~\text{and}~x_{\varphi(v)}\in x+B(\theta, \frac{1}{2u_{1}}), \forall v \geq v_{0}.
	\end{equation}
	Let $v \geq v_{0}$ be given. For any $n\in \mathbb{N}$, $m\in \mathbb{N}$ and $i\in \{1,\cdots,k_{m}^{n}\}$, by (\ref{eq.Kakutani1}) we have
	\begin{align}
		I_{(\varphi(v)=n)\wedge d_{m,i}^{n}}x_{\varphi(v)}&=I_{(\varphi(v)=n)\wedge d_{m,i}^{n}}\sum_{l=1}^{\infty}I_{(\varphi(v)=l)}x_{l}\nonumber\\
		&=I_{(\varphi(v)=n)\wedge d_{m,i}^{n}}x_{n}\nonumber\\
		&=I_{(\varphi(v)=n)\wedge d_{m,i}^{n}}x_{m}^{n}\nonumber\\
		&\in I_{(\varphi(v)=n)\wedge d_{m,i}^{n}} (x_{m,i}^{n}+ B(\theta, \frac{1}{n})),\nonumber
	\end{align}
	which, combined with (\ref{eq.Kakutani4}), implies that
	\begin{align}
		I_{(\varphi(v)=n)\wedge d_{m,i}^{n}} x_{m,i}^{n}&\in I_{(\varphi(v)=n)\wedge d_{m,i}^{n}} (x_{\varphi(v)}+ B(\theta, \frac{1}{n}))\nonumber\\
		&=I_{(\varphi(v)=n)\wedge d_{m,i}^{n}} (x_{\varphi(v)}+ B(\theta, \frac{1}{\varphi(v)}))\nonumber\\
		&\subseteq I_{(\varphi(v)=n)\wedge d_{m,i}^{n}}(x+ B(\theta, \frac{1}{2u_{1}}) +B(\theta,\frac{1}{\varphi(v)}))\nonumber\\
		&\subseteq I_{(\varphi(v)=n)\wedge d_{m,i}^{n}}(x+ B(\theta, \frac{1}{u_{1}})).\nonumber
	\end{align}
	Furthermore, by Lemma \ref{lemm.basic} and (\ref{eq.Kakutani3}) we have 
	\begin{align}
		I_{(\varphi(v)=n)\wedge d_{m,i}^{n}}y_{m,i}^{n}&\in I_{(\varphi(v)=n)\wedge d_{m,i}^{n}}F(x_{m,i}^{n})\nonumber\\
		&\subseteq I_{(\varphi(v)=n)\wedge d_{m,i}^{n}}F[G\cap (x+B(\theta, \frac{1}{u_{1}}))] \nonumber\\
		&\subseteq I_{(\varphi(v)=n)\wedge d_{m,i}^{n}}(F(x)+ B(\theta,\frac{1}{2u})) \nonumber
	\end{align}
	for any $n\in \mathbb{N}$, $m\in \mathbb{N}$ and $i\in \{1,\cdots,k_{m}^{n}\}$. Arbitrarily choose $y_{*}\in F(x)$, since $F(x)+B(\theta,\frac{1}{2u}) -y_{*}$ is a $\sigma$-stable set with $\theta\in F(x)+B(\theta,\frac{1}{2u}) -y_{*}$, we have 
	\begin{align}
		I_{(\varphi(v)=n)\wedge d_{m,i}^{n}} (y_{m,i}^{n}-y_{*}) 
		&\in I_{(\varphi(v)=n)\wedge d_{m,i}^{n}} (F(x)+B(\theta,\frac{1}{2u}) -y_{*}) \nonumber \\
		&\subseteq F(x)+B(\theta,\frac{1}{2u}) -y_{*}\nonumber
	\end{align}
	for any $n\in \mathbb{N}$, $m\in \mathbb{N}$ and $i\in \{1,\cdots,k_{m}^{n}\}$. Furthermore, $F(x)+B(\theta,\frac{1}{2u}) -y_{*}$ is also $L^{0}$-convex, by (\ref{eq.Kakutani2}) we have 
	\begin{align}
			&x_{\varphi(v)}-y_{*}\nonumber\\
			&=\sum_{n=1}^{\infty}I_{(\varphi(v)=n)}(x_{n}-y_{*})\nonumber\\
			&=\sum_{n=1}^{\infty}I_{(\varphi(v)=n)}\sum_{m=1}^{\infty}I_{a_{m}^{n}}(x_{m}^{n}-y_{*})\nonumber\\
		    &=\sum_{n=1}^{\infty}I_{(\varphi(v)=n)} \sum_{m=1}^{\infty}I_{a_{m}^{n}}\sum_{i=1}^{k_{m}^{n}}
		    \frac{\alpha_{m,i}^{n}(x_{m}^{n}) }{ \sum_{j=1}^{k_{m}^{n}} \alpha_{m,j}^{n} (x_{m}^{n}) } (y_{m,i}^{n}-y_{*})\nonumber\\
		    &=\sum_{n=1}^{\infty}I_{(\varphi(v)=n)} \sum_{m=1}^{\infty}I_{a_{m}^{n}}\sum_{i=1}^{k_{m}^{n}}
		    \frac{I_{(\varphi(v)=n)\wedge a_{m}^{n}}\alpha_{m,i}^{n}(x_{m}^{n}) }{ \sum_{j=1}^{k_{m}^{n}} \alpha_{m,j}^{n} (x_{m}^{n}) } (y_{m,i}^{n}-y_{*})\nonumber\\
		    &=\sum_{n=1}^{\infty}I_{(\varphi(v)=n)} \sum_{m=1}^{\infty}I_{a_{m}^{n}}\sum_{i=1}^{k_{m}^{n}}
		    \frac{I_{(\varphi(v)=n)\wedge d_{m,i}^{n}}\alpha_{m,i}^{n}(x_{m}^{n}) }{ \sum_{j=1}^{k_{m}^{n}} \alpha_{m,j}^{n} (x_{m}^{n}) } (y_{m,i}^{n}-y_{*})\nonumber\\
		    &=\sum_{n=1}^{\infty}I_{(\varphi(v)=n)} \sum_{m=1}^{\infty}I_{a_{m}^{n}}\sum_{i=1}^{k_{m}^{n}}
		    \frac{\alpha_{m,i}^{n}(x_{m}^{n}) }{ \sum_{j=1}^{k_{m}^{n}} \alpha_{m,j}^{n} (x_{m}^{n}) } I_{(\varphi(v)=n)\wedge d_{m,i}^{n}}(y_{m,i}^{n}-y_{*})\nonumber\\
		    &\in \sum_{n=1}^{\infty}I_{(\varphi(v)=n)} \sum_{m=1}^{\infty}I_{a_{m}^{n}}\sum_{i=1}^{k_{m}^{n}}
		    \frac{\alpha_{m,i}^{n}(x_{m}^{n})) }{ \sum_{j=1}^{k_{m}^{n}} \alpha_{m,j}^{n} (x_{m}^{n})) } (F(x)+B(\theta,\frac{1}{2u}) -y_{*})\nonumber\\
		    &\subseteq F(x)+B(\theta,\frac{1}{2u})-y_{*},\nonumber
	\end{align}
	namely, 
	\begin{equation}\label{eq.Kakutani5}
		x_{\varphi(v)}\in  F(x)+B(\theta,\frac{1}{2u}).
	\end{equation}
	Since (\ref{eq.Kakutani5}) holds for any $v\geq v_{0}$ and $\{x_{\varphi(v)},v\in L^{0}(\mathcal{F},\mathbb{N})\}$ converges in $\mathcal{T}_{c}$ to $x$, we have
	$$x\in [F(x)+B(\theta,\frac{1}{2u})]^{-}_{c}\subseteq F(x)+B(\theta,\frac{1}{2u})+B(\theta,\frac{1}{2u})\subseteq F(x)+B(\theta,\frac{1}{u}).$$
\end{proof}

    %#########################################################################
    %#########################################################################
    \section{Concluding remarks and open problems}\label{sec.4}
    %#########################################################################
    %#########################################################################

    \par 
    As the notion of a $\mathcal{T}_{c}$-upper semicontinuous set-valued mapping is of topological nature, the notion of a $\mathcal{T}_{\varepsilon,\lambda}$-upper semicontinuous set-valued mapping can be introduced in a similar way.

    \begin{definition}
    Let $(E_{1},\|\cdot\|_{1})$ and $(E_{2},\|\cdot\|_{2})$ be two $RN$ modules over $\mathbb{K}$ with base $(\Omega,\mathcal{F},P)$, $X\subseteq E_{1}, Y\subseteq E_{2}$  two nonempty sets and  $F:X\rightarrow 2^{Y}\setminus \{\emptyset\}$ a set-valued mapping. $F$ is said to be $\mathcal{T}_{\varepsilon,\lambda}$-upper semicontinuous at $x_{0}\in X$ if for every $\mathcal{T}_{\varepsilon,\lambda}$-neighborhood $U$ of $F(x_{0})$, there is a $\mathcal{T}_{\varepsilon,\lambda}$-neighborhood $V$ of $x_{0}$ such that  $F(V\cap X)\subseteq U$ (equivalently, the upper inverse $F^{+}(U)$ is a $\mathcal{T}_{\varepsilon,\lambda}$-neighborhood of $x_{0}$ in $X$). Furthermore,  $F$ is said to be $\mathcal{T}_{\varepsilon,\lambda}$-upper semicontinuous on $X$ if it is $\mathcal{T}_{\varepsilon,\lambda}$-upper semicontinuous at every point $x\in X$.
    \end{definition}

   \par 
   The choice between $\mathcal{T}_{c}$-upper semicontinuity and $\mathcal{T}_{\varepsilon,\lambda}$-upper semicontinuity may lead to two different possible versions of the random Kakutani fixed point theorem. We have established Theorem \ref{thm.Kakutani}, which can be regarded as the $\mathcal{T}_{c}$-version, whereas the $\mathcal{T}_{\varepsilon,\lambda}$-version has been not yet  established in this paper, namely, Problem \ref{prob.1} below is still open.

   \begin{problem}\label{prob.1}
   	Let $(E,\|\cdot\|)$ be an $RN$ module over $\mathbb{K}$ with base $(\Omega,\mathcal{F},P)$, $G$ a random sequentially compact $L^{0}$-convex subset of $E$ and $F:G\rightarrow 2^{G}\backslash\{\emptyset\}$ a $\sigma$-stable $\mathcal{T}_{\varepsilon,\lambda}$-upper semicontinuous mapping such that $F(x)$ is closed and $L^{0}$-convex for each $x\in G$. Does $F$ have a fixed point?
   \end{problem}

   \par 
   In comparing the two versions of the random Kakutani fixed point theorem, a natural question arises as to whether one version implies the other. In the case of single-valued mappings, the noncompact Schauder fixed point theorem \cite[Theorem 2.12]{GWXYC2025} can be regarded as the $\mathcal{T}_{c}$-version, which in fact includes the $\mathcal{T}_{\varepsilon,\lambda}$-version as a special case. More precisely, Guo et al. introduced the notion of a random sequentially continuous single-valued mapping (see part (5) of \cite[Definition 2.11]{GWXYC2025}) and proved that a $\sigma$-stable single-valued mapping $f$ is random sequentially continuous if and only if it is $\mathcal{T}_{c}$-continuous (see \cite[Lemma 4.3]{GWXYC2025}). Therefore,  a $\sigma$-stable $\mathcal{T}_{\varepsilon,\lambda}$-continuous single-valued mapping is $\mathcal{T}_{c}$-continuous since any $\mathcal{T}_{\varepsilon,\lambda}$-continuous single-valued mapping defined on a $\sigma$-stable set is necessarily random sequentially continuous.

   \par 
   Proposition \ref{prop.usc} below is a known result in classical set-valued analysis.

   \begin{proposition}[\cite{AF2009}]\label{prop.usc}
   	Let $(X,\mathcal{T}_{X})$ and $(Y,\mathcal{T}_{Y})$ be two topological spaces and $F:X\rightarrow 2^{Y}\setminus \{\emptyset\}$ be a set-valued mapping. Then the
   	following statements are equivalent:
   	\begin{enumerate}[(1)]
   		\item $F$ is upper semicontinuous on $X$;
   		\item For any $x\in X$, any net $\{x_{\alpha},\alpha\in \Lambda\}$ in $X$ converges to $x$ and any $O_{Y}\in \mathcal{T}_{Y}$ with $F(x)\subseteq O_{Y}$,
   		there exists $\alpha_{0}\in \Lambda$ such that $F(x_{\alpha})\subseteq O_{Y}$ for any $\alpha\in \Lambda$ with $\alpha\geq \alpha_{0}$.
   	\end{enumerate}
   \end{proposition}

   \par 
    Guided by Proposition \ref{prop.usc} and in comparison with \cite[Definition 2.11]{GWXYC2025}, one can naturally introduce the following Definition \ref{defn.rsusc}.

   \begin{definition}\label{defn.rsusc}
   Let $(E_{1},\|\cdot\|_{1})$ and $(E_{2},\|\cdot\|_{2})$ be two $RN$ modules over $\mathbb{K}$ with base $(\Omega,\mathcal{F},P)$, $X\subseteq E_{1}, Y\subseteq E_{2}$ two nonempty $\sigma$-stable sets and $F:X\rightarrow 2^{Y}\setminus \{\emptyset\}$ a set-valued mapping. $F$ is said to be random sequentially upper semicontinuous at $x_{0}\in X$ if $X$ is $\sigma$-stable and if for any sequence $\{x_{n},n\in \mathbb{N}\}$ in $X$ converges in $\mathcal{T}_{\varepsilon,\lambda}$ to $x_{0}$ and any $\mathcal{T}_{\varepsilon,\lambda}$-neighborhood $V$ of $F(x_{0})$, there exist a random subsequence $\{x_{n_{k}},k\in \mathbb{N}\}$ of $\{x_{n},n\in \mathbb{N}\}$ and $k_{0}\in \mathbb{N}$ such that
   $$F(x_{n_{k}})\subseteq V, \forall k\geq k_{0}.$$
   Further, $F$ is said to be random sequentially upper semicontinuous on $X$ if $F$ is random sequentially upper semicontinuous at any point in $X$. 
   \end{definition}

   \par 
   Since $\mathcal{T}_{\varepsilon,\lambda}$ is metrizable, it is easy to check that a $\mathcal{T}_{\varepsilon,\lambda}$-upper semicontinuous set-valued mapping defined on a $\sigma$-stable set is necessarily random sequentially upper semicontinuous. Then, in the spirit of \cite{GWXYC2025}, investigating Problem \ref{prob.2} below may serve as a crucial step toward resolving Problem \ref{prob.1}, which also clarifies the relative strength of the two versions of the random Kakutani fixed point theorem.

   \begin{problem}\label{prob.2}
   	Let $(E_{1},\|\cdot\|_{1})$ and $(E_{2},\|\cdot\|_{2})$ be two $RN$ modules over $\mathbb{K}$ with base $(\Omega,\mathcal{F},P)$, $X\subseteq E_{1}$ and $Y\subseteq E_{2}$ two $\sigma$-stable sets, and $F:X\rightarrow 2^{Y}\setminus \{\emptyset\}$ a $\sigma$-stable set-valued mapping. Is it true that $F$ is random sequentially upper semicontinuous if and only if $F$ is $\mathcal{T}_{c}$-upper semicontinuous?
   \end{problem}

%%    Bibliographies can be prepared with BibTeX using amsplain,
%%    amsalpha, or (for "historical" overviews) natbib style.

%\bibliographystyle{amsplain}


\begin{thebibliography}{99}
	
	

	
	
	
	\bibitem{Aub1979}
	{\scshape Aubin, J.}
	Mathematical methods of game and economic theory.
	Stud. Math. Appl., Vol.7.
	{\em North-Holland, Amsterdam-New York}, 1979.
	\mrev{0556865} (83a:90005),
	\zbl{0452.90093}.
	

	
	
	
	
	\bibitem{AF2009}
	{\scshape Aubin, J.; Frankowska, H.}
	Set-valued analysis.
	Modern. Birkhäuser Class.
	{\em Birkh\"{a}user,  Boston}, 2009.
	\mrev{2458436},
	\zbl{1168.49014}.
	\doi{10.1007/978-0-8176-4848-0}. 
	
	
	
	
	
	\bibitem{BK1950}
	{\scshape Bohnenblust, H.F.; Karlin, S.}
	On a theorem of Ville. Contributions to the Theory of Games, pp. 155--160.
	Ann. of Math. Stud., no. 24.
	{\em Princeton University Press, Princeton}, 1950.
	\mrev{0041415} (12,844c),
	\zbl{0041.25701},
	\doi{10.1515/9781400881727-014}.  
 
	
	
	
	\bibitem{Bor1985}
	{\scshape Border, K.C.}
	Fixed Point Theorems with Applications to Economics and Game Theory. 
	{\em Cambridge University Press, Cambridge}, 1985.
	\mrev{0790845} (86j:90002),
	\zbl{0558.47038},
	\doi{10.1017/CBO9780511625756}.  
	
	
	
	
	
	\bibitem{Brouwer1912}
	{\scshape Brouwer, L.E.J.}
	\"{U}ber Abbildung von Mannigfaltigkeiten.
	{\em Math. Ann.}
	\textbf{71} (1912), no. 4, 97--115. 
	\mrev{1511678},
	\zbl{42.0417.01},
	\doi{10.1007/BF01456812}.  
	
	
	

	
	
	\bibitem{BPS2023}
	{\scshape Bru\`{e}, E.; Pasqualetto, E.; Semola, D.}
	Rectifiability of the reduced boundary for sets of finite perimeter over $\operatorname{RCD}(K,N)$ spaces.
	{\em J. Eur. Math. Soc.}
	\textbf{25} (2023),  no. 2, 413--465. 
	\mrev{4556787},
	\zbl{1523.49051},
	\doi{10.4171/JEMS/1217}.  
	
	
	
	
	
	
	\bibitem{CLPV2025}
	{\scshape Caputo, E.; Lu\v{c}i\'{c}, M.; Pasqualetto, E.; Vojnovi\'{c}, I.}
	On the integration of $L^{0}$-Banach $L^{0}$-modules and its applications to vector calculus on $\operatorname{RCD}$ spaces.
	{\em Rev. Mat. Complut.}
	\textbf{38} (2025), no. 1, 149--182. 
	\mrev{4859192},
	\zbl{07985695},
	\doi{10.1007/s13163-024-00491-8}.  
	
	
	
	
	
	\bibitem{CGP2025}
	{\scshape Caputo, E.; Gigli, N.; Pasqualetto, E.}
	Parallel transport on non-collapsed $\operatorname{RCD}(K,N)$ spaces.
	{\em J. Reine Angew. Math.}
	\textbf{819} (2025), 135--204. 
	\mrev{4856997},
	\zbl{07982962},
	\doi{10.1515/crelle-2024-0082}.  
	
	
	
	
	
	
	
	\bibitem{Cel1970}
	{\scshape Cellina, A.}
	Multivalued differential equations and ordinary differential equations. 
    {\em SIAM J. Appl. Math.}
    \textbf{18} (1970), 533--538.
    \mrev{0276537} (43\#2281),
    \zbl{0191.38802},
    \doi{10.1137/0118046}.  
	
	
	
	
	\bibitem{Deb1959}
	{\scshape Debreu, G.}
	Theory of value: an axiomatic analysis of economic equilibrium.
	{\em John Wiley $\&$ Sons, New York}, 1959.
	\mrev{0110571}(22\#1447),
	\zbl{0193.20205}.
	
	
	
	\bibitem{DS1958} 
	{\scshape Dunford, N.; Schwartz, J.T.}
	Linear Operators (I): General Theory. 
	Pure Appl. Math., Vol. 7.
	{\em Interscience, New York}, 1958.
	\mrev{0117523} (22\# 8302),
	\zbl{0084.10402}.
	
	
	
	
	
	\bibitem{Fan1952} 
	{\scshape Fan, K.}
	Fixed point and minimax theorems in locally convex topological linear spaces.
	{\em Proc. Nat. Acad. Sci. U.S.A.}
	\textbf{38} (1952), 121--126.
	\mrev{0047317} (13,858d),
	\zbl{0047.35103},
	\doi{10.1073/pnas.38.2.121}.	
	
	
	
	
	
	\bibitem{FKV2009}
	{\scshape Filipovi\'{c}, D.; Kupper, M.; Vogelpoth, N.}
	Separation and duality in locally $L^{0}$-convex modules.
	{\em J. Funct. Anal.} 
	{\bf256} (2009), no. 12, 3996--4029.
	\mrev{2521918} (2011b:46081),
	\zbl{1180.46055},
	\doi{10.1016/j.jfa.2008.11.015}.
	
	
	
	
	
	\bibitem{Gigli2018}
	{\scshape Gigli, N.}
	Nonsmooth differential geometry --- an approach tailored for spaces with Ricci curvature bounded from below. 
	{\em Mem. Amer. Math. Soc.} 
	{\bf251} (2018), 1196.  
	\mrev{3756920},
	\zbl{1404.53056},
	\doi{10.1090/memo/1196}.
	
	
	
	
	
	\bibitem{GLP2025}
	{\scshape Gigli, N.; Lu\v{c}i\'{c}, D.; Pasqualetto, E.}
	Duals and pullbacks of normed modules. 
	{\em Israel J. Math.} 
	\textbf{267} (2025), no. 2, 821--866. 
    \mrev{4931495},
    \zbl{08069426},
    \doi{10.1007/s11856-025-2725-2}.
	
	
	
	
	
	
	\bibitem{Gli1952}
	{\scshape Glicksberg, I.L.}
	A further generalization of the Kakutani fixed theorem, with application to Nash equilibrium points.
	{\em Proc. Amer. Math. Soc.} 
	{\bf 3} (1952), 170--174.
	\mrev{0046638}(13,764g), 
	\zbl{0046.12103}, 
	\doi{10.2307/2032478}.
	
	
	
	
	\bibitem{GD2003}
	{\scshape Granas, A.; Dugundji, J.}
	Fixed Point Theory. 
	{\em Springer-Verlag, New York}, 2003.
	\mrev{1987179}, 
	\zbl{1025.47002}, 
	\doi{10.1007/978-0-387-21593-8}.
	
	
	
	
	
	\bibitem{Guo1992}
	{\scshape Guo, T.X.}
	Random metric theory and its applications. 
	Ph.D thesis, Xi'an Jiaotong University, China, 1992.
	
	
	
	
	
	\bibitem{Guo1993}
	{\scshape Guo, T.X.}
	A new approach to probabilistic functional analysis. 
	In: Proceedings of the first China Postdoctoral Academic Conference, pp. 1150--1154. 
	{\em The China National Defense and Industry Press, Beijing}, 1993.
	
	
	
	

	
	
	\bibitem{Guo2008}%7
	{\scshape Guo, T.X.}
	The relation of Banach-Alaoglu theorem and Banach-Bourbaki-Kakutani-\v{S}mulian theorem in complete random normed modules to stratification structure.
	{\em Sci. China Ser. A} 
	{\bf51} (2008), no. 9, 1651--1663.
	\mrev{2426061} (2009h:46146),
	\zbl{1167.46049},
	\doi{10.1007/s11425-008-0047-6}.
	
	
	
	
	
	
	
	\bibitem{Guo2010}
	{\scshape Guo, T.X.}
	Relations between some basic results derived from two kinds of topologies for a random locally convex module.
	{\em J. Funct. Anal.} 
	{\bf258} (2010), no. 9, 3024--3047.  
	\mrev{2595733} (2011a:46073),
	\zbl{1198.46058},
	\doi{10.1016/j.jfa.2010.02.002}.

	
	
	
	
	\bibitem{Guo2024}
	{\scshape Guo, T.X.}
	Optimization of conditional convex risk measures. 
	In: Cheng S.Y., Li S., Yau S.T. Proceedings of the International Congress of Chinese Mathematicians (Beijing, 2019), Vol. 1, pp. 347--371.
	{\em International Press of Boston, Somerville}, 2024.
	\zbl{08086388}.
	
	
	
	
	
	\bibitem{GMT2024}
	{\scshape Guo, T.X.; Mu, X.H.; Tu, Q.}
	The relations among the notions of various kinds of stability and their applications.
	{\em Banach J. Math. Anal.} 
	{\bf 18} (2024), no. 3, 42.
	\mrev{4750396},
	\zbl{07871743},
	\doi{10.1007/s43037-024-00354-w}.
	
	
	
	
	
	
	\bibitem{GWXYC2025}
	{\scshape Guo, T.X.; Wang, Y.C.; Xu, H.K.; Yuan, G.; Chen, G.}
	A noncompact Schauder fixed point theorem in random normed modules and its applications.
	{\em Math. Ann.} 
	\textbf{391} (2025), no. 3, 3863--3911.   
	\mrev{4865231},
	\zbl{07986071},
	\doi{10.1007/s00208-024-03017-1}.
	
	
	
	\bibitem{GWYZ2020}
	{\scshape Guo, T.X.; Wang, Y.C.; Yang, B.X.; Zhang, E.X.} 
	On $d$-$\sigma$-stability in random metric spaces and its applications. 
	{\em J. Nonlinear Convex Anal.} 
	\textbf{21} (2020), no. 6, 1297--1316.
	\mrev{4157373},
	\zbl{1476.46001}.
	
	
	
	\bibitem{GZZ2014}
	{\scshape Guo, T.X.; Zhao, S.E.; Zeng, X.L.} 
	The relations among the three kinds of conditional risk measures. 
	{\em Sci. China Math.} 
	\textbf{57} (2014), no. 8, 1753--1764. 
	\mrev{3229236},
	\zbl{1316.46004},
	\doi{10.1007/s11425-014-4840-0}.
 
	
	
	
	
	
	\bibitem{GZWG2020} 
	{\scshape Guo, T.X.; Zhang, E.X.; Wang, Y.C.; Guo, Z.C.}
	Two fixed point theorems in complete random normed modules and their applications to backward stochastic equations. 
	{\em J. Math. Anal. Appl.} 
	\textbf{483} (2020), no. 2, 123644.  
	\mrev{4037576}, 
	\zbl{1471.60086}, 
	\doi{10.1016/j.jmaa.2019.123644}.
	
	

	
	
	
	\bibitem{GZWYZ2017}
	{\scshape Guo, T.X.; Zhang, E.X.; Wu, M.Z.; Yang, B.X.; Yuan, G.; Zeng, X.L.} 
	On random convex analysis. 
	{\em J. Nonlinear Convex Anal.} 
	\textbf{18} (2017), no. 11, 1967--1996. 
	\mrev{3734154}, 
	\zbl{1396.46056}.	
	
	
	
 
	
	

	
	
	
	\bibitem{KS2001} 
	{\scshape Kabanov, Y.; Stricker, C.}
	A teacher's note on no-arbitrage criteria. 
	In: Az\'{e}ma, J., \'{E}mery, M., Ledoux, M., Yor M.
	S\'{e}minaire de Probabilit\'{e}s, XXXV, pp. 149--152. 
	Lecture Notes in Math., Vol. 1755.
	{\em Springer, Berlin}, 2001.
	\mrev{1837273} (2002a:60003),
	\zbl{0982.60032}, 
	\doi{10.1007/b76885}.
	
	
	
	
	
	\bibitem{Kakutani1941}
	{\scshape Kakutani, S.}
	A generalization of Brouwer's fixed point theorem.
	{\em Duke Math. J.} 
	\textbf{8} (1941), 457--459. 
	\mrev{0004776} (3,60c),
	\zbl{0061.40304}, 
	\doi{10.1215/S0012-7094-41-00838-4}.
	
	
	
	
	
	
	\bibitem{Kop1989}
	{\scshape Koppelberg, S.}
	General theory of Boolean algebras. 
	Handbook of Boolean Algebras, Vol. 1.
	{\em North-Holland, Amsterdam}, 1989.
	\mrev{0991565} (90k:06002),
	\zbl{0671.06001}.
	
	
	
	
	
	\bibitem{LO1955}
	{\scshape Lasota, A.; Opial, Z.}
	An application of the Kakutani-Ky Fan theorem in the theory of ordinary differential equations. 
	{\em Bull. Acad. Polon. Sci. Sér. Sci. Math. Astronom. Phys.} 
	\textbf{13} (1965), 781--786.
	\mrev{0196178} (33\#4370),
	\zbl{0151.10703}.
	
	
	
	
	
	\bibitem{LP2019}
	{\scshape Lu\v{c}i\'{c}, D.; Pasqualetto, E.}
	The Serre-Swan theorem for normed modules.
	{\em Rend. Circ. Mat. Palermo} 
	{\bf 68} (2019), no. 2, 385--404.  
	\mrev{4148753},
	\zbl{1433.46030},
	\doi{10.1007/s12215-018-0366-6}.
	
	
	
	
	
	
	\bibitem{LPV2024}
	{\scshape Lu\v{c}i\'{c}, M.; Pasqualetto, E.; Vojnovi\'{c}, I.}
	On the reflexivity properties of Banach bundles and Banach modules.
	{\em Banach J. Math. Anal.}
	{\bf18} (2024), no. 1, 7.   
	\mrev{4678838},
	\zbl{1529.18004},
	\doi{10.1007/s43037-023-00315-9}.
	
	
   
   
   \bibitem{MTG2025}
   {\scshape Mu, X.H.; Tu, Q.; Guo, T.X.; Xu, H.K.}
   Common fixed point theorems for a commutative family of nonexpansive mappings in complete random normed modules.
   {\em J. Fixed Point Theory Appl.}
   {\bf27} (2025), no. 3, 78.   
   \mrev{4944102},
   \zbl{08084676},
   \doi{10.1007/s11784-025-01231-1}.
   

	
	
	\bibitem{Nash1951'}
	{\scshape Nash, J.}
	Equilibrium points in $n$-person games.
	{\em Proc. Nat. Acad. Sci. U.S.A.} 
	\textbf{36} (1950), 48--49.
	\mrev{0031701}(11,192c),
	\zbl{0036.01104},
	\doi{10.1073/pnas.36.1.48}.
	
	
	
	
	
	\bibitem{Nikaid1968}
	{\scshape Nikaid\^{o}, H.} 
	Convex Structures and Economic Theory.
	Mathematics in Science and Engineering, Vol. 51. 
	{\em Academic Press, New York-London}, 1968.
	\mrev{0277233}(43\#2970),
	\zbl{0172.44502}.
	
	

	
	\bibitem{Ponosov1988}
	{\scshape Ponosov, A.}
	A fixed point method in the theory of stochastic differential equations.
	{\em Soviet Math. Dokl.} 
	{\bf 37} (1988), no. 2, 426--429.  
	\mrev{0936492} (89g:60189),
	\zbl{0693.47043}.
	
	
	
	
	\bibitem{Ponosov2021}
	{\scshape Ponosov, A.}
	A counterexample to the stochastic version of the Brouwer fixed point theorem.
	{\em Russ. Univ. Rep. Math.} 
	{\bf 26} (2021), no. 134, 143--150.   
	\zbl{1495.47083},
	\doi{10.20310/2686-9667-2021-26-134-143-150}.
	
	
	
	
	\bibitem{Ponosov2022}
	{\scshape Ponosov, A.}
	A fixed point theorem for local operators with applications to stochastic equations.
	arXiv:2208.00674, 2022.
	\doi{10.48550/arXiv.2208.00674}.
 
	
	
	
	
	
	\bibitem{SS2005} 
	{\scshape Schweizer, B.; Sklar, A.}
	Probabilistic Metric Spaces.
	{\em Elsevier/North-Holland, New York}, 1983;
	{\em Dover Publications, New York}, 2005.
	\mrev{0790314} (86g:54045),
	\zbl{0546.60010}.
	
	
	
	
	
	
	\bibitem{SV1991} 
	{\scshape Shapley, L.; Vohra, R.}
	On Kakutani's fixed point theorem, the K-K-M-S theorem and the core of a balanced game. 
	{\em Econ. Theory}  
	\textbf{1} (1991), 108--116. 
	\mrev{1095155} (92c:90135),
	\zbl{0807.90141},
	\doi{10.1007/BF01210576}.
	
 
	
	
	
	
	\bibitem{TMG2024}
	{\scshape Tu, Q.; Mu, X.H.; Guo, T.X.}
	The random Markov-Kakutani fixed point theorem in a random locally convex module. 
	{\em New York J. Math.} 
	\textbf{30} (2024), 1196--1219.
	\mrev{4791060},
	\zbl{07922018}.
	
	 
	
	
	
	\bibitem{TMGC2025}
	{\scshape Tu, Q.; Mu, X.H.; Guo, T.X.; Chen, G.}
	A new complete proof of the random Brouwer fixed point theorem and its implied consequences of unification. 
	arXiv:2507.08521v4, 2025. 
	\doi{10.48550/arXiv.2507.08521}.
	
	
	
	\bibitem{WG2024}
	{\scshape Wang, Y.C.; Guo, T.X.}
	The fixed point theorems and invariant approximations for random nonexpansive mappings in random normed modules.
	{\em J. Math. Anal. Appl.} 
	\textbf{531} (2024), no. 1, 127796.
	\mrev{4649341},
	\zbl{07852378}.
	\doi{10.1016/j.jmaa.2023.127796}.
	

\end{thebibliography}
\end{document}